\documentclass[11pt,twoside,a4paper]{article}

\usepackage[utf8]{inputenc}
\usepackage[T1]{fontenc}
\usepackage{amsmath,amssymb,dsfont}
\numberwithin{equation}{section}
\usepackage{microtype}
\usepackage{lmodern}
\usepackage{graphicx,tikz,pgfplots}
\usepackage{tikz-3dplot}
\usetikzlibrary{arrows,shapes,fadings,intersections,decorations.pathreplacing,positioning}
\graphicspath{{images/}}
\pgfplotsset{compat=newest}
\usepackage[hyperref,amsmath,thmmarks]{ntheorem}
\usepackage{aliascnt}
\usepackage[a4paper,centering,bindingoffset=0cm,marginpar=2cm,margin=2.5cm]{geometry}
\usepackage[pagestyles]{titlesec}
\usepackage[font=footnotesize,format=plain,labelfont=sc,textfont=sl,width=0.75\textwidth,labelsep=period]{caption}
\usepackage{comment}
\usepackage{siunitx}
\usepackage{mathtools}

\usetikzlibrary{quotes}
\usetikzlibrary{angles}
\usetikzlibrary{babel}

\usetikzlibrary{decorations.pathreplacing}
\usetikzlibrary{decorations.markings}

\usepackage{mathabx}
\tikzset{
    >=stealth',
    punkt/.style={
           rectangle,
           rounded corners,
           draw=black, very thick,
           text width=6.5em,
           minimum height=2em,
           text centered},
    pil/.style={
           ->,
           thick,
           shorten <=2pt,
           shorten >=2pt,}
}
\tikzstyle{block} = [rectangle, rounded corners, minimum width= 3cm, minimum height=1cm, text centered, draw=black, fill=blue!20,]

\tikzstyle{invisbleblock} = [minimum width= 3em, minimum height=1cm, text centered, ]

\tikzstyle{decision} = [diamond, minimum width=3cm, minimum height=1cm, text centered, draw=black, fill=orange!20]

\tikzstyle{arrow} = [thick,->,>=stealth]
\tikzstyle{line} = [thick,-]

\usepackage{dsfont,bbm}

\usepackage[boxed]{algorithm2e}

\usepackage{enumerate}

\usepackage{subcaption}

\usepackage[backend=biber,maxnames=10,backref=true,hyperref=true,giveninits=true,safeinputenc]{biblatex}
\ifdefined\biberExtract
\else
\fi

\DefineBibliographyStrings{english}{%
	backrefpage = {cited on page},
	backrefpages = {cited on pages},
}

\title{Fourier reconstruction for diffraction tomography \\ of an object rotated into arbitrary orientations}
\author{Clemens Kirisits$^1$\\
{\footnotesize\href{mailto:clemens.kirisits@univie.ac.at}{clemens.kirisits@univie.ac.at}} 
\and Michael Quellmalz$^2$\\
{\footnotesize\href{mailto:quellmalz@math.tu-berlin.de}{quellmalz@math.tu-berlin.de}}
\and Monika Ritsch-Marte$^3$\\ 
{\footnotesize\href{mailto:monika.ritsch-marte@i-med.ac.at}{monika.ritsch-marte@i-med.ac.at}}
\and Otmar Scherzer$^{1,4}$\\
{\footnotesize\href{mailto:otmar.scherzer@univie.ac.at}{otmar.scherzer@univie.ac.at}}
\and Eric Setterqvist$^4$\\
{\footnotesize\href{mailto:eric.setterqvist@ricam.oeaw.ac.at}{eric.setterqvist@ricam.oeaw.ac.at}}
\and Gabriele Steidl$^2$\\
{\footnotesize\href{mailto:steidl@math.tu-berlin.de}{steidl@math.tu-berlin.de}}    }
\date{\today}

\DeclareFieldFormat[report]{title}{``#1''}
\DeclareFieldFormat[book]{title}{``#1''}
\AtEveryBibitem{\clearfield{url}}
\AtEveryBibitem{\clearfield{note}}

\titleformat{\section}[block]{\large\sc\filcenter}{\thesection.}{0.5ex}{}[]
\titleformat{\subsection}[runin]{\bf}{\thesubsection.}{0.5ex}{}[.]

\usepackage[pdftex,colorlinks=true,linkcolor=blue,citecolor=green,urlcolor=blue,bookmarks=true,bookmarksnumbered=true,pdfusetitle]{hyperref}
\hypersetup
{
    pdfauthor={C. Kirisits, M. Quellmalz, M. Ritsch-Marte, O. Scherzer, E. Setterqvist, G. Steidl},
}

\newpagestyle{headers}
{
	\headrule
	\sethead[\footnotesize\thepage][\footnotesize\sc C.~Kirisits, M.~Quellmalz, M.~Ritsch-Marte, O.~Scherzer, E.~Setterqvist, G.~Steidl][]{}{\footnotesize\sc Fourier reconstruction for diffraction tomography 
  }{\footnotesize\thepage}
	\setfoot{}{}{}
}
\newtheorem{lemma}{Lemma}[section]

\newaliascnt{proposition}{lemma}
\newtheorem{proposition}[proposition]{Proposition}
\aliascntresetthe{proposition}

\newaliascnt{corollary}{lemma}
\newtheorem{corollary}[corollary]{Corollary}
\aliascntresetthe{corollary}

\newaliascnt{theorem}{lemma}
\newtheorem{theorem}[theorem]{Theorem}
\aliascntresetthe{theorem}

\theorembodyfont{\normalfont}
\newaliascnt{definition}{lemma}
\newtheorem{definition}[definition]{Definition}
\aliascntresetthe{definition}

\newaliascnt{assumption}{lemma}

\aliascntresetthe{assumption}

\newaliascnt{notation}{lemma}

\aliascntresetthe{notation}

\newaliascnt{example}{lemma}
\newtheorem{example}[example]{Example}
\aliascntresetthe{example}

\newaliascnt{experiment}{lemma}

\aliascntresetthe{experiment}

\newaliascnt{remark}{lemma}
\newtheorem{remark}[remark]{Remark}
\aliascntresetthe{remark}

\theoremstyle{nonumberplain}
\theoremseparator{:}
\theoremheaderfont{\normalfont\itshape}

\theoremsymbol{\ensuremath{\square}}
\newtheorem{proof}{Proof}

\newcommand{\N}{\mathds{N}}
\newcommand{\R}{\mathds{R}}
\newcommand{\C}{\mathds{C}}

\newcommand{\abs}[1]{\left|#1\right|}
\newcommand{\norm}[1]{\left\|#1\right\|}

\newcommand{\e}{\mathrm e}
\let\ii\i
\renewcommand{\i}{\mathrm i}

\newcommand{\sphere}{\mathbb{S}^2}

\newcommand{\bn}{\mathbf{n}}
\newcommand{\br}{{\bf r}}
\newcommand{\bs}{{\bf s}}

\newcommand{\by}{{\bf y}}
\newcommand{\bk}{{\bf k}}

\newcommand{\bv}{{\bf v}}
\newcommand{\be}{{\bf e}}
\newcommand{\bz}{{\bf z}}
\newcommand{\bh}{{\bf h}}

\newcommand{\ui}{u^{\mathrm{inc}}}
\newcommand{\uborn}{u^{\mathrm{Born}}}

\newcommand{\utot}{u^{\mathrm{tot}}}
\newcommand{\phitot}{\varphi^{\mathrm{tot}}}

\newcommand{\us}{u^{\mathrm{sca}}}
\newcommand{\phir}{\varphi^{\mathrm{Rytov}}}
\newcommand{\phii}{\varphi^{\mathrm{inc}}}

\newcommand{\phis}{\varphi^{\mathrm{sca}}}

\newcommand{\tT}{\top}
\newcommand{\rs}{r_{\mathrm{s}}}

\newcommand{\supp}{\operatorname{supp}}

\newcommand{\ktran}{\mathcal{F}}

\begin{document}

\maketitle
\thispagestyle{empty}
\begin{center}
	\parbox[t]{18em}{\footnotesize
		\hspace*{-1ex}$^1$Faculty of Mathematics\\
		University of Vienna\\
		Oskar-Morgenstern-Platz 1\\
		A-1090 Vienna, Austria}
	\hfil
	\parbox[t]{18em}{\footnotesize
		\hspace*{-1ex}$^2$Institute of Mathematics\\
		Technical University Berlin\\
		Straße des 17. Juni 136 \\
		D-10623 Berlin, Germany}
	\\[1ex]
	\parbox[t]{18em}{\footnotesize
		\hspace*{-1ex}$^3$Institute of Biomedical Physics \\
		Medical University of Innsbruck \\
		Müllerstraße 44\\
		A-6020 Innsbruck, Austria}
	\hfil
	\parbox[t]{18em}{\footnotesize
		\hspace*{-1ex}$^4$Johann Radon Institute for Computational\\
		\hspace*{1em}and Applied Mathematics (RICAM)\\
		Altenbergerstraße 69\\
		A-4040 Linz, Austria} 

\end{center}
\begin{abstract} 
In this paper, we study the mathematical imaging problem of optical 
diffraction tomography (ODT) for the scenario of a microscopic rigid particle rotating in a trap created, 
for instance, by acoustic or optical forces. 
Under the influence of the inhomogeneous forces the  particle carries out a time-dependent smooth, but  complicated motion described by a set of affine transformations. The rotation of the particle enables one to record optical images from a wide range of angles, which largely eliminates the ``missing cone problem'' in optics.
This advantage, however, comes at the price that the rotation axis in this scenario is not fixed, but continuously undergoes some variations, and that the rotation angles are not equally spaced, which is in contrast to standard tomographic reconstruction assumptions.
In the present work, we assume that the time-dependent motion parameters are known, 
and that the particle's  scattering potential is compatible with making the first order Born or Rytov approximation.
We prove a Fourier diffraction theorem 
and derive novel backprojection formulae for the reconstruction of the scattering potential, 
which depends on the refractive index distribution inside the object, 
taking its complicated motion into account.
This provides the basis for solving the ODT problem with an efficient non-uniform discrete Fourier transform.
\end{abstract}

\section{Introduction} \label{sec:intro}
In \emph{optical  diffraction tomography} (ODT), see, for instance  \cite{ArrHeb97,Arr99,ArrScho09,sung2009optical,jin2017tomographic},
the three-dimensional (3D) refractive index distribution of an object is constructed from optical measurements, i.e.~from intensity images
or from interferometric data, taken from different angles. The illumination directions are varied, for instance by active scanning or by means of a lenslet array as in Fourier ptychography~\cite{tian2014multiplexed}. Alternatively, the object is embedded in a gel and rotated while the illumination direction is kept fixed. In either case, the light propagation directions utilized to  solve the inverse problem from a set of  images  recorded from these directions are known. Moreover, typically the viewing directions are regularly distributed around a prevalent direction, which means that the viewing angles are restricted to a certain hardware-related  interval. Depending on the width of this interval, the sampling in the Fourier plane remains incomplete, which leads to artefacts in the reconstruction, such as the well-know ``missing cone artefact'' in optics~\cite{krauze2020optical,LimLeeJinShiLee15,SunDas11}, which are exactly the geometrical figures plotted in \autoref{fig:coverage}. 

The present paper is motivated by the context of carrying out ODT on a trapped particle which is held in place by optical or acoustic forces~\cite{dholakia2020comparing}. We assume that these forces can be employed to induce rotations of the trapped particle in a controlled way,
e.g.~by using holographic optical tweezers with several spots or by standing ultrasound waves. 
This generally allows to view the particle from a wider range of directions 
than possible with illumination scans on a fixed object, and thus leads to more complete sampling in Fourier space 
and consequently to fewer artifacts. 
However, this comes at the price that the viewing directions are then not as regularly spaced as normally the case. 
Even more importantly, in optical or acoustic trapping the particle itself is not completely immobilized and the locally acting forces are typically inhomogeneous, so that the particle undergoes a time-dependent  smooth, but complicated motion described by a set of affine transformations. 
In a real experimental situation, a video would be recorded and the parameters of the affine transformations, i.e.~translation vector  and rotation axis and angle at a given time,  would have to be extracted at every sampling time by some suitable method.
In this work we assume these parameters to be \emph{known}. 
Apart from this, we assume that the particle can be described as a rigid body and that its center of mass  is fixed, which means we do not have any translations but only rotation around the center of mass.

The concrete model, which we study in this paper, is based on a series of assumptions:
\begin{enumerate}
	\item[(i)] 
	In a lens-less imaging configuration, the object is probed with \emph{coherent} light, assumed as plane waves.  
	The optical \emph{field} (not just the intensity)  is measured, either in reflection or in transmission. This means that the full optical amplitude and phase information has to be captured, for example interferometrically. Alternatively, we can consider data recorded by intensity measurements, where the phase has been determined by a phase-retrieval algorithm under suitable constraints~\cite{fienup1987reconstruction}.
	\item[(ii)] 
	The scattering properties of the probe only \emph{slightly deviate} 
	from the background, meaning that linearized models assuming Born's or Rytov's approximation are valid. 
	If these simplifications cannot be made, methods from \emph{full waveform inversion} have to be considered, see, for instance \cite{VirAsnBroMetRib14}. This is not pursued here.
	\item[(iii)] 
	Certain assumptions for the propagation of the light through the object have to be made, i.e.~on the \emph{forward model}: 
	If the properties of  object and  optical set-up allow the imaging to be approximately given by geometric optics, then, mathematically, the problem becomes analogous to inverting the 3D X-ray transform, see \cite{ChoFanOhBadDas07}. 
	In this case, the optical image resembles a projection image, and optical projection tomography can be used~\cite{sharpe2002optical}. This is, for instance, fulfilled in low numerical aperture imaging of biological samples with sufficient amplitude contrast, with large structures on the scale of the optical wavelength and with limited refractive index contrast.
	For imaging with a  higher numerical aperture objective, and/or for samples with small structures diffracting the light beams, deviations from geometrical optics are to be expected. In this situation diffraction models, 
see, for instance the work of Devaney \cite{Dev82,Dev83,Dev84} and \cite{MueSchuGuc15_report}, 
more accurately describing the propagation of the light through the sample, 
need to be considered. 
These models will be investigated in this paper. Please note that all of these approaches still assume that first order Born or Rytov are valid, which means that the object cannot be strongly refracting or scattering.
	\item[(iv)] 
	We assume that the motion of the particle has been estimated beforehand. We mention our work~\cite{ElbRitSchSchm20} for retrieving the motion parameters, which is applicable if the imaging of the sample is well approximated by \emph{projections}, as described above. 
	For motion estimation also the algorithms from single particle cryogenic electron microscropy (cryo-EM), see \cite{Hee87,Gon88a,CoiShkSigSin10,SinShk11,WanSinWen13}, can be used, which also serves as a prerequisite step for 3D-visualization also based on inverting the X-ray transform.  
\end{enumerate}

In this paper, we present an algorithm for 3D visualization of a single rigid object rotating around its center of mass. 
The mathematical model describing the optical experiment is a diffraction tomography model based on Born's or Rytov's simplifications. One challenge is the alignment of the exciting plane waves with the measurement plane. This particular setup therefore requires new mathematical research: 
\begin{enumerate}
	\item[(i)] In contrast to the references on diffraction tomography mentioned above, in the envisaged tomography on levitated or trapped objects the measurement data is not uniformly sampled. To deal with such measurements, 
	we use reconstruction formulae in the k-space, also known as Fourier-space, see \autoref{sec:backprop}.
	\item[(ii)] Born and Rytov approximations are also used in various kind of other applications 
    such as seismics and ultrasound tomography, see \cite{Dev81,Dev82,Dev83,Dev84,NatWue01,Nat15}. 
    The frequency coverage on the reconstruction of the 3D Fourier transform of the scattering potential of the probe which is given by the measurement data, similarly as in  \autoref{fig:coverage} below, 
		have been observed there first  \cite{WuTok87,Mor89}. 
	\item[(iii)] For the numerical reconstruction, we propose the use of the inverse \emph{Nonequispaced Discrete Fourier Transform} (NDFT), which can deal with the irregular motion considered here. For other regularization methods (such as variational ones) we refer to \cite{LimLeeJinShiLee15,SunDas11}.
	\item[(iv)] From a mathematical perspective it is important to consider the describing diffraction equations and backpropagation formulas in a rigorous distributional setting. Mathematically this sheds some new light on the reconstruction formulae.
\end{enumerate}

The outline of this paper is as follows: 
In \autoref{sec:exp}, we introduce the mathematical setting of diffraction tomography, and formulate the basic model based on Born's or Rytov's approximation for wave propagation. 
In \autoref{sec:back}, we state a rigorous proof of the Fourier diffraction theorem, which builds the foundation of our reconstruction formulae.
Then, in \autoref{sec:backprop}, we derive a backpropagation reconstruction formula in k-space, which can take into account arbitrary (uncontrolled) rotations of the sample. 
\autoref{sec:disc} discusses the discretization of the backpropagation formula from the preceding section as well as an alternative reconstruction method based on the inverse NDFT. 
\autoref{sec:num_results} shows numerical reconstructions, comparing the backpropagation formula with the inverse NDFT.
The appendix \autoref{appendix} provides background information on distributions and Fourier analysis
and some of the rather technical proofs.

\section{Conceptual Experiment} \label{sec:exp}
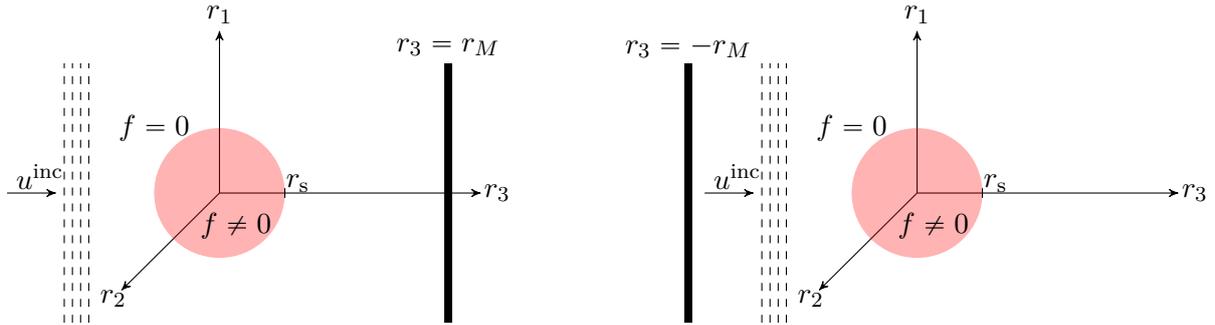
\begin{figure}
	\begin{center}
		\begin{tikzpicture}[scale=0.43,
		>=stealth',
		pos=.8,
		photon/.style={decorate,decoration={snake,post length=1mm}}
		]
		
		\draw[->] (2,0)--(10,0)  ;
		\draw[->] (2,0)--(2,5)  ;
		\draw[->] (2,0)--(-1,-3)  ;
		\node at (10.5,0) {$r_3$};
		\node at (2,5.5) {$r_1$};
		\node at (-1.25,-3.25) {$r_2$};
		
		\fill[red,opacity=0.3] (2,0) circle (2);
		\node at (0,2) {$f = 0$};
		\node at (2.5,-1) {$f \neq 0$};
		\draw (4,-0.15) -- (4,0.15);
		\node at (4.4,.3) {$\rs$};
		
		\draw[dashed] (-2,-4) -- (-2,4);
		\draw[dashed]  (-2.25,-4) -- (-2.25,4);
		\draw[dashed]  (-2.5,-4) -- (-2.5,4);
		\draw[dashed] (-2.75,-4) -- (-2.75,4);
		\draw[->] (-4.5,0) -- (-3,0);
		\node at (-3.5,.5) {$\ui$};	
		
		
		\draw[line width = 3pt] (9,-4) -- (9,4);
		\node at (9,4.5) {$r_3 = r_M$};
		
		\end{tikzpicture}
		\hspace{1cm}
		\begin{tikzpicture}[scale=0.43,
		>=stealth',
		pos=.8,
		photon/.style={decorate,decoration={snake,post length=1mm}}
		]
		
		\draw[->] (2,0)--(10,0)  ;
		\draw[->] (2,0)--(2,5)  ;
		\draw[->] (2,0)--(-1,-3)  ;
		\node at (10.5,0) {$r_3$};
		\node at (2,5.5) {$r_1$};
		\node at (-1.25,-3.25) {$r_2$};
		
		\fill[red,opacity=0.3] (2,0) circle (2);
		\node at (0,2) {$f = 0$};
		\node at (2.5,-1) {$f \neq 0$};
		\draw (4,-0.15) -- (4,0.15);
		\node at (4.4,.3) {$\rs$};
		
		\draw[dashed] (-2,-4) -- (-2,4);
		\draw[dashed]  (-2.25,-4) -- (-2.25,4);
		\draw[dashed]  (-2.5,-4) -- (-2.5,4);
		\draw[dashed] (-2.75,-4) -- (-2.75,4);
		\draw[->] (-4.5,0) -- (-3,0);
		\node at (-3.5,.5) {$\ui$};	
		
		
		\draw[line width = 3pt] (-5,-4) -- (-5,4);
		\node at (-5,4.5) {$r_3 = -r_M$};
		
		\end{tikzpicture}
	\end{center}
\caption{Conceptual setup. The support of $f$ lies entirely in $\mathcal{B}_{\rs}$.
\emph{Left:} Transmission imaging. \emph{Right:} Reflection imaging. }
	\label{fig:trans}
\end{figure}

In this section, we describe the  experimental situation we have in mind when developing our tomographic reconstruction method for arbitrary object rotations  and formulate the mathematical wave propagation models.
A schematic overview of the set-up is given in \autoref{fig:trans}.

The object we want to image tomographically is illuminated by a plane wave $\ui$ 
which propagates in  direction $\be_3 = (0,0,1)^\tT$ with wave number $k_0$, that is
\begin{equation} \label{eq:plane_wave}
\ui(\br) = \e^{\i k_0 r_3}.
\end{equation}
Note that the wavelength $\lambda$ of $\ui$ is related to the wave number via $k_0 = \frac{2\pi}{\lambda}$. 
The object is assumed to be enclosed by the open ball $\mathcal{B}_{\rs} \subset \R^3$ 
centered at $\mathbf{0}$ with radius $\rs$. 
In order to generate multiple illuminations of the object, it is rotated around its center of mass 
which is fixed at $\mathbf{0}$. 
The incident wave $\ui$ induces a scattered wave $\us$ which is recorded in a plane at a distance from the object 
at  $r_3 = r_M > \rs$  for transmission imaging and at $r_3 = - r_M < -\rs$ for reflection imaging, respectively.

By $n(\br)$ we denote the refractive index of the object and by $n_0$ the constant refractive index of the background, respectively. 
The scattering properties of the object are characterized by the function 
 \begin{equation} \label{eq:refractive}
   k(\br) = k_0 \frac{n(\br)}{n_0}\,, \qquad \text{ for all } \br \in \R^3.
 \end{equation}
Note that $k$ differs from $k_0$ only in $\mathcal{B}_{\rs}$.
The function
\begin{equation}
f(\br) = k^2(\br)-k_0^2
\end{equation}
is referred to as the 
\emph{scattering potential}~\cite{Wol69} and will be the quantity which we set out to reconstruct from the measurements of the scattered waves. By construction we have
 \begin{equation} \label{eq:supp1}
  \supp (f)\subseteq\mathcal{B}_{\rs}\subset (-\rs,\rs)^3.
 \end{equation}
The total field $\utot=\ui+\us$ satisfies the \emph{reduced wave equation}
 \begin{equation} \label{eq:tot}
 (\Delta  + k^2) \utot = 0,
 \end{equation}
while $\ui$ in turn fulfills the \emph{Helmholtz equation}
\begin{equation} \label{eq:tot_const}
(\Delta + k_0^2) \ui = 0.
\end{equation}

Next we will consider two simplifications of \autoref{eq:tot}, the Born and Rytov approximations. While their underlying assumptions and physical validity are different, they may be analyzed within the same mathematical framework. More on the background and comparisons of the approximations can be found in \cite[Chap.~6]{KakSla01}.
%
\subsection{Born approximation} \label{ss:Born}
Inserting \autoref{eq:tot_const} in \autoref{eq:tot} together with some rearranging gives
\begin{equation} \label{eq:totII}
-(\Delta + k_0^2)\us = f(\us+\ui). 
\end{equation}
Neglecting $\us$ in the right-hand side of \autoref{eq:totII} (assuming that $\us$ is small in comparison with $\ui$), we obtain the \emph{Born approximation}:
 \begin{equation} \label{eq:Born}
  -(\Delta + k_0^2) \uborn = f\ui.
 \end{equation}
%
\subsection{Rytov approximation}
We make the ansatz
\begin{equation} \label{eq:exp}
  \utot= \ui + \us= \e^{\phitot}, \quad \ui =\e^{\i k_0 r_3}= \e^{\phii}
  \text{ and } \phitot= \phii + \phis.
\end{equation}
The calculations below should be considered formal. From \autoref{eq:exp} it
follows that
\begin{equation} \label{eq:expsca}
\begin{aligned}
\us &= \utot - \ui = \e^{\phitot} - \e^{\phii} = \e^{\phii} \left( \e^{\phis} - 1\right).
\end{aligned}
\end{equation}
We have
\begin{equation} \label{helper1}
\nabla  \utot = \utot \nabla \phitot 
\quad \mathrm{and} \quad  
\Delta \utot =  \utot\left(\Delta \phitot + (\nabla \phitot)^2 \right),
\end{equation}
where $\left(\nabla\phitot\right)^2=\sum_{i=1}^{3}\left(\frac{\partial\phitot}{\partial r_i}\right)^2$
and similarly for $\ui$.
Then the Helmholtz \autoref{eq:tot_const} and
the reduced wave \autoref{eq:tot} 
can be rewritten as
\begin{align}
-k_0^2 &= \Delta \phii + (\nabla \phii)^2 , \nonumber\\
-k^2 &= \Delta \phitot + (\nabla \phitot)^2  
= \Delta \phii  + \Delta \phis + \left( \nabla \phii + \nabla \phis\right)^2\nonumber \\
&= \Delta \phii + (\nabla \phii)^2 + \Delta \phis + 2  \nabla \phii \cdot \nabla \phis +\left(\nabla\phis\right)^2,\nonumber\\
-f &=  \Delta \phis + 2  \nabla \phii \cdot \nabla \phis + (\nabla \phis)^2. \label{helper2}
\end{align}
For establishing the Rytov approximation, we consider
$$
(\Delta + k_0^2)(\ui \phis) = \phis \Delta \ui + 2 \ui \, \nabla \phii \cdot \nabla \phis + \ui \Delta \phis + k_0 \ui \phis,
$$
which by \autoref{eq:tot_const} becomes
$$
(\Delta + k_0^2)(\ui \phis) =  \left( 2 \nabla \phii \cdot \nabla \phis + \Delta \phis \right)  \ui.
$$
and by \autoref{helper2} finally
\begin{equation} \label{eq:totphase}
-(\Delta +k_0^2)(\ui \phis)=\left(f + \left(\nabla\phis\right)^2\right)\ui.
\end{equation}
Neglecting $\left(\nabla\phis\right)^2$ in \autoref{eq:totphase}, we obtain the \emph{Rytov approximation} $\phir$:
\begin{equation} \label{eq:Rytov}
  -(\Delta +k_0^2)(\ui \phir)=f\ui.
\end{equation}
Comparing the Born and Rytov approximations, we find their relation
\begin{equation*} 
	\uborn = \ui\phir.
\end{equation*}
Therefore, it is sufficient to concentrate on the Born approximation. For simplicity of notation, 
we set $\uborn=u$ from now on.

\section{Fourier diffraction theorem} \label{sec:back}
The Fourier diffraction theorem, see \cite[Sect.\ 6.3]{KakSla01}, \cite[Thm.\ 3.1]{NatWue01} or \cite{Wol69}, is the basis for reconstruction formulae in diffraction tomography, as it relates the 2D Fourier transform of the measurements to the 3D Fourier transform of the scattering potential. In this section, we establish a version of the Fourier diffraction theorem that suits the experimental setup described in \autoref{sec:exp}.

Under Born's approximation, the relation between the scattered wave $u$ and the scattering potential $f$ is governed by \autoref{eq:Born}. We assume further that $u$ satisfies the \emph{Sommerfeld radiation condition}
\begin{equation*}
\lim_{s\to\infty} \max_{\norm\br=s} \norm{\br} \left(\frac{\partial }{\partial \br} u(\br)-\i k_0 u(\br)\right)=0,
\end{equation*}
where $\frac{\partial}{\partial \br}$ denotes the directional derivative. Physically speaking, $u$ is an outgoing wave.

To formulate the diffraction theorem, we need the following notation.
We set
\begin{equation} \label{eq:kappa}
\kappa=\kappa(k_1,k_2) \coloneqq
\begin{cases}
\sqrt{k_0^2-k_1^2-k_2^2}, & k_1^2+k_2^2\leq k_0^2,\\
\i\sqrt{k_1^2+k_2^2-k_0^2}, & k_1^2+k_2^2>k_0^2.
\end{cases}
\end{equation}
Furthermore, $H_{r_3}: \R^3 \rightarrow \R$ denotes the \emph{Heaviside function} 
in the third coordinate centered at $r_3$, that is, 
$$H_{r_3}(s_1,s_2,s_3) 
\coloneqq
\left\{
\begin{array}{ll}
0& \mathrm{if} \; s_3 < r_3,\\      
1& \mathrm{otherwise}. 
\end{array}
\right.
$$
Finally,  we denote by $\ktran_{1,2}$ the \emph{partial Fourier transform} which is taken with respect to the first two components,
by $\mathcal{D}^\prime(\R^3)$ the space of distributions and by $S^\prime(\R^3)$,
the space of tempered distributions. 
Further background material on Fourier transforms, convolutions and distributions is given in \autoref{ss:ktran} and \autoref{ss:Fpartial} in the appendix.
Then we can establish the following theorem whose proof is given in \autoref{sec:fdt}.

\begin{theorem} \label{thm:fourierdiffraction}
	Let $k_0>0$ and $g\in L^p(\R^3)$, $p>1$, with $\supp(g)\subset\mathcal{B}_{r}$ for some $r>0$. Suppose that the function $u$ is the solution of
	\begin{equation} \label{eq:Born2}
	\Delta u + k_0^2 u = -g
	\end{equation}
which satisfies the Sommerfeld radiation condition. 
Then, we can identify the distribution $\mathcal F_{1,2}u$ almost everywhere with the following locally integrable function
	\begin{equation} \label{eq:F12u}
	\ktran_{1,2}u(k_1,k_2,r_3)=\sqrt{\frac{\pi}{2}}\frac{\i}{\kappa} \left(\e^{\i \kappa r_3} \ktran \left((1-H_{r_3})g \right)(k_1,k_2,\kappa) 
	+ \e^{-\i \kappa r_3}\ktran \left(H_{r_3}g\right)(k_1,k_2,-\kappa) \right).
	\end{equation}
\end{theorem}

\begin{remark}
\begin{itemize}
\item[\textrm{(i)}] By the fundamental lemma of variational calculus of du Bois--Reymond, 
see \cite[Lem. 3.2]{Gru09}, \autoref{eq:F12u} is equivalent
with equality of both functions considered as distributions in $\mathcal{D^\prime}(\R^3)$.
Note that the identification of \autoref{eq:F12u} does not hold, in general, for $\ktran_{1,2}u$ 
as an element of  $S^\prime(\R^3)$ since the right-hand side is not guaranteed 
to be polynomially bounded as $\norm{(k_1,k_2,r_3)}\to\infty$, compare with \autoref{ex:L1loc}.
\item[(ii)]	Recall further that $\kappa$ is imaginary for $k_1^2+k_2^2>k_0^2$. In this case, in \autoref{eq:F12u}, we have to consider the analytic continuations of $\ktran\left((1-H_{r_3})g \right)$ and $\ktran \left(H_{r_3}g\right)$ to $\C^3$.
\end{itemize}
\end{remark}

The theorem implies the following Fourier diffraction result for our setting, see \autoref{fig:trans}.
\begin{corollary}\label{thm:fdt}
Assume that
\begin{enumerate}
	\item[\rm{(i)}] the scattering potential of the probe is given by $f \in L^p$, $p>1$, where $\supp(f)\subset\mathcal{B}_{\rs}$, $0 < \rs <r_M$,
	\item[\rm{(ii)}] the incident field is a plane wave $\ui(\br) = \e^{\i k_0 r_3}$,
	\item[\rm{(iii)}] the Born approximation is valid for the scattered field $u$ and $u$ satisfies the Sommerfeld radiation condition,
	\item[\rm{(iv)}] the scattered field $u$ is measured at the plane $r_3 = r_M$ (transmission imaging) or $r_3 = -r_M$ (reflection imaging).
\end{enumerate}
Then
	\begin{equation} \label{eq:recon}
		\ktran_{1,2}u(k_1,k_2,\pm r_M)=\sqrt{\frac{\pi}{2}}\frac{\i \e^{\i\kappa r_M}}{\kappa}\ktran f(k_1,k_2,\pm\kappa-k_0)
	\end{equation}
	for all $k_1,k_2 \in \R$ satisfying $k_1^2 + k_2^2 \neq k_0^2$.
\end{corollary}
\begin{proof}
Assuming validity of the Born approximation, the scattered wave $u$ satisfies \autoref{eq:Born}. According to \autoref{thm:fourierdiffraction} with $g=f\ui$, we obtain
\begin{equation*}
	\begin{aligned}
		\ktran_{1,2}u(k_1,k_2,r_3) 
		&=	
		\frac{\i \sqrt{\pi}}{\kappa  \sqrt{2}} \left(\e^{\i \kappa r_3} \ktran \left((1-H_{r_3})f \ui \right)(k_1,k_2,\kappa) 
		+ \e^{-\i \kappa r_3}\ktran \left(H_{r_3}f \ui \right)(k_1,k_2,-\kappa) \right) \\
		&=
		\frac{\i \sqrt{\pi}}{\kappa  \sqrt{2}} \left(\e^{\i \kappa r_3} \ktran \left((1-H_{r_3})f \right)(k_1,k_2,\kappa-k_0) 
		+ \e^{-\i \kappa r_3}\ktran \left(H_{r_3}f \right)(k_1,k_2,-\kappa-k_0) \right),
	\end{aligned}
\end{equation*}
where we have exploited the specific form of the incident wave $\ui$. 
Finally, for transmission imaging, we have $r_3 = r_M$, so that $H_{r_M}f = 0$ and $(1-H_{r_M})f = f$. 
Similarly, for reflection imaging, where $r_3=-r_M$, we obtain $(1-H_{-r_M})f = 0$ and $H_{-r_M}f = f$.
\end{proof}

\begin{remark}\label{rem:hemisphere}
Even though, mathematically, \autoref{eq:recon} holds as long as $k_1^2 + k_2^2 \neq k_0^2$, physically speaking the spatial frequencies with $k_1^2 + k_2^2 > k_0^2$ do not contribute to the measurements. Therefore, without rotation of the object, the measurements in both transmission and reflection imaging provide access to the scattering potential $f$ on a hemisphere
\begin{equation*}
	\{(k_1,k_2,\pm\kappa-k_0)^\tT : k_1,k_2 \in \R,\, k_1^2 + k_2^2 < k_0^2 \}
\end{equation*}
in k-space. The two hemispheres are depicted in \autoref{fig:coverage}.
\end{remark}
\begin{figure}[h]
\includegraphics[width=.5\textwidth]{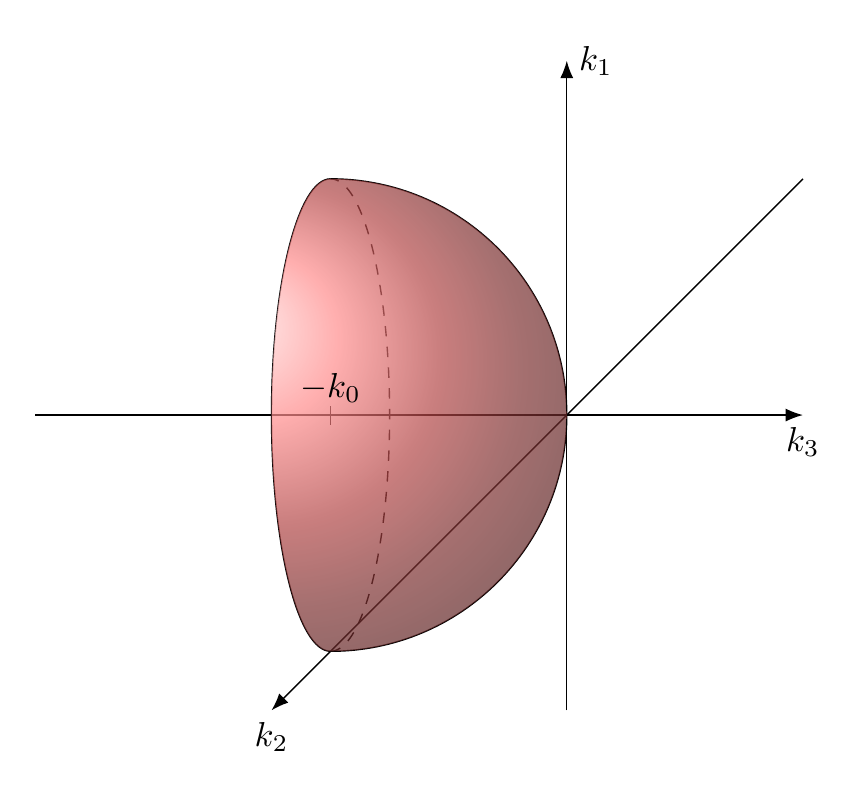}%
\includegraphics[width=.5\textwidth]{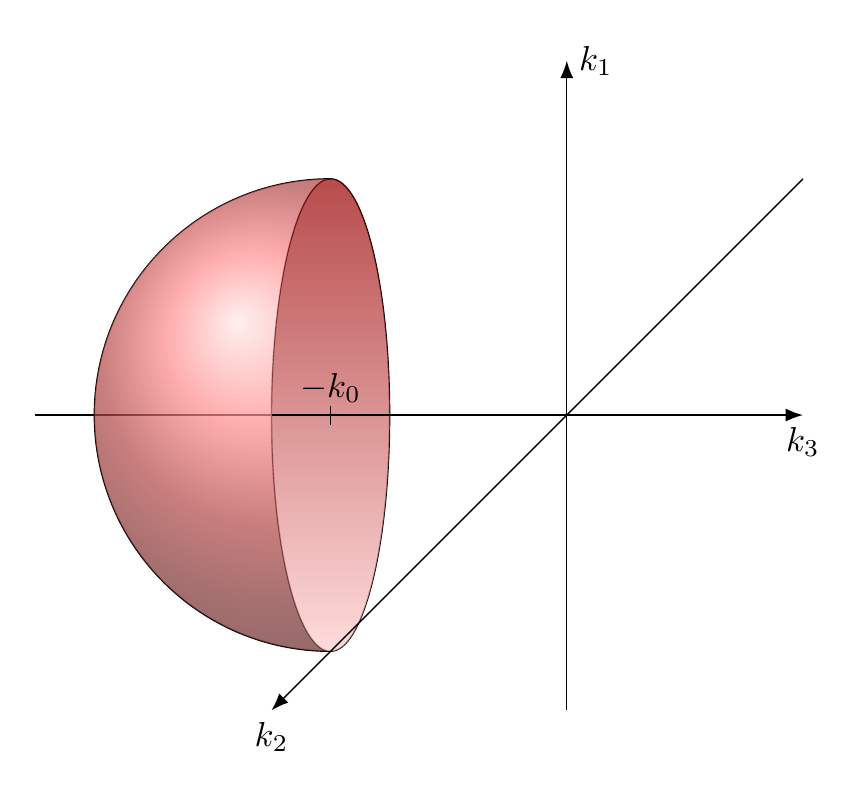}
\caption{Accessible points in k-space for transmission imaging (left) and reflection imaging (right).}
\label{fig:coverage}
\end{figure}

\section{Backpropagation formulae} \label{sec:backprop}
According to the Fourier diffraction theorem, illumination of the object from a single direction provides access to the scattering potential on a \emph{surface} in k-space. If the probe is rotated continuously, however, we can obtain knowledge of $\ktran f$ in a \emph{volume} $\mathcal{Y},$ see \autoref{fig:cover2} for an illustration of such a volume. Consequently, an approximation of the scattering potential that incorporates all the available information can be found by means of Fourier inversion
$$ f(\br) \approx (2\pi)^{-\frac32} \int_\mathcal{Y} \e^{\i \by \cdot \br} \ktran f(\by)\, d\by.$$
The reconstruction formula for our setting is made precise in \autoref{thm:reconstruction} and \autoref{thm:jacobian} below.

We assume that the scattering object undergoes a rotation with varying rotation axis. Therefore, if the function $\bn=\bn(t) : [0,L] \to \mathbb{S}^2$ describes the orientation of this axis and $\alpha = \alpha(t) : [0,L] \to \R$ is the corresponding angle, then the rotation can be represented by the matrix 
\begin{equation*}
R_{\bn,\alpha}^\tT \coloneqq 
\begin{pmatrix}
n_1^2(1-\mathrm{c}) + \mathrm{c} & n_1 n_2 (1-\mathrm{c}) - n_3 \mathrm{s} & 
n_1n_3(1-\mathrm{c}) + n_2 \mathrm{s}\\
n_1n_2(1-\mathrm{c}) + n_3 \mathrm{s} & n_2^2 (1-\mathrm{c}) + \mathrm{c} & 
n_2n_3(1-\mathrm{c}) - n_1 \mathrm{s}\\
n_1n_3(1-\mathrm{c}) - n_2\mathrm{s} & n_2 n_3 (1-\mathrm{c}) + n_1 \mathrm{s} & 
n_3^2(1-\mathrm{c}) + \mathrm{c}
\end{pmatrix},
\end{equation*}
where $\bn=(n_1,n_2,n_3)^\top$, $\mathrm{c} \coloneqq \cos \alpha$ and $\mathrm{s} \coloneqq \sin \alpha $. Note that we follow the convention of describing the rotation of the object by the transpose $R_{\bn,\alpha}^\tT$. Therefore, the scattering potential of the rotated object is given by $f \circ R_{\bn,\alpha}$.

Next, denote by $u_t$, $0\le t \le L$, the wave scattered by the rotated potential $f \circ R_{\bn(t),\alpha(t)}$. Under Born's approximation, it satisfies
\begin{equation*}
	\Delta u_t + k_0^2 u_t = - \ui f \circ R_{\bn(t),\alpha(t)}.
\end{equation*}
The full set of measurements in the transmission and reflection setup, respectively, is then given by
\begin{equation}\label{eq:measurements}
	u_t(r_1,r_2,\pm r_M), \quad r_1,r_2 \in \R, \quad 0 \le t \le L,
\end{equation}
and, according to \autoref{thm:fdt}, it is related to the scattering potential $f$ via
\begin{equation} \label{eq:recon_n}
	\ktran_{1,2}u_t(k_1,k_2,\pm r_M)= \sqrt{\frac{\pi}{2}}\frac{\i \e^{\i\kappa r_M}}{\kappa}
						\ktran f \left(R_{\bn(t),\alpha(t)}(k_1,k_2,\pm\kappa-k_0)^\tT \right).
\end{equation}
Next, let 
\begin{equation*}
	\mathcal{U} \coloneqq \{(k_1,k_2,t) \in \R^3 : k_1^2 + k_2^2 < k_0^2, \, 0\le t \le L \}
\end{equation*}
be the set where \autoref{eq:recon_n} is valid and can be used for reconstruction purposes, recall \autoref{rem:hemisphere}.
Moreover, the map that traces out the accessible domain in k-space is denoted by
\begin{equation}\label{eq:T}
	T _{\pm}: \mathcal{U} \to \mathbb{R}^3, \quad T_\pm(k_1,k_2,t) \coloneqq R_{\bn(t),\alpha(t)}(k_1,k_2, \pm \kappa - k_0)^\top.
\end{equation}
In the reconstruction formula below we have to take into account the number of times a point $\by$ in k-space is covered by $T_{\pm}$. This number, sometimes referred to as \emph{Banach indicatrix} of $T_\pm$\footnote{In Computerized Tomography (CT) the Banach indicatrix is called Crofton symbol (see for instance \cite{LouWebThe08}).}, will be denoted by $\operatorname{Card}(T_{\pm}^{-1}(\by))$, where $\operatorname{Card}(A)$ is the cardinality of a set $A$.
Finally, the approximation to $f$ we wish to reconstruct is
\begin{equation} \label{eq:fbp}
	f^{\pm}_{\mathrm{bp}}(\br) \coloneqq (2\pi)^{-\frac32} \int_{T_{\pm}(\mathcal{U})} \e^{\i \by \cdot \br} \ktran f(\by)\, d\by.
\end{equation}
The set $T_{\pm}(\mathcal{U})$ will be referred to as the \emph{frequency coverage} or \emph{k-space coverage} of the experimental setup.

\begin{theorem}\label{thm:reconstruction}
Let the assumptions of \autoref{thm:fdt} be satisfied. In addition, assume that $\alpha \in C^1[0,L]$ and $\bn\in C^1([0,L],\sphere)$. Then, for all $\br \in \R^3,$
\begin{equation}\label{eq:reconstruction}
	f^{\pm}_{\mathrm{bp}}(\br) = \frac{-\i}{2 \pi^2} \int_{\mathcal{U}} \e^{\i T_{\pm}(k_1,k_2,t) \cdot \br}  \ktran_{1,2}u_t(k_1,k_2,\pm r_M) \frac{\kappa \e^{-\i \kappa r_M} \abs{\nabla T_{\pm}(k_1,k_2,t)}}{\operatorname{Card}(T_{\pm}^{-1}(T_{\pm}(k_1,k_2,t)))} d(k_1,k_2,t),
\end{equation}
where $|\nabla T_{\pm}|$ is the magnitude of the Jacobian determinant of $T_{\pm}$.
\end{theorem}
\begin{proof}
First, we show that $\operatorname{Card}(T_{\pm}^{-1})$ is finite at almost every point in the range of $T_\pm$. By \autoref{thm:jacobian}, we have $|\nabla T_{\pm}| \in L^1(\mathcal{U})$. Applying a change of variables that takes into account the potential noninjectivity of $T_\pm$, see \cite[Thm.\ 5.8.30]{Bog07}, 
gives 
$\int_{T_{\pm}(\mathcal{U})} \operatorname{Card}(T_{\pm}^{-1}) \, d\by= \int_{\mathcal{U}} \abs{\nabla T_{\pm}}\, d\by$. It follows that $\operatorname{Card}(T_{\pm}^{-1}(\by)) < \infty$ for almost every $\by \in T_\pm(\mathcal{U})$.

Next, we can write
\begin{equation*}
	f^{\pm}_{\mathrm{bp}}(\br)
		= (2\pi)^{-\frac32} \int_{T_{\pm}(\mathcal{U})} \e^{\i \by \cdot \br} \ktran f(\by) \frac{\operatorname{Card}(T_{\pm}^{-1}(\by))}{\operatorname{Card}(T_{\pm}^{-1}(\by))}  \, d\by,
\end{equation*}
since $\operatorname{Card}(T_{\pm}^{-1})$ vanishes nowhere and is finite almost everywhere. Applying \cite[Thm.\ 5.8.30]{Bog07} once again, where the function to be integrated is $\by \mapsto \e^{\i \by \cdot \br} \ktran f(\by)/\operatorname{Card}(T_{\pm}^{-1}(\by))$, gives
\begin{equation} \label{eq:approx2g}
	f^{\pm}_{\mathrm{bp}}(\br)
		= (2\pi)^{-\frac32} \int_{\mathcal{U}} \e^{\i T_{\pm}(k_1,k_2,t) \cdot \br} \ktran f(T_{\pm}(k_1,k_2,t)) \frac{\abs{\nabla T_{\pm}(k_1,k_2,t)}}{\operatorname{Card}(T_{\pm}^{-1}(T_{\pm}(k_1,k_2,t)))}  \, d(k_1,k_2,t).
\end{equation}
The assertion now follows after using \autoref{eq:recon_n} to express $\ktran f$ in terms of the measurements.
\end{proof}

In order to actually use the reconstruction formula \eqref{eq:reconstruction}, it remains to calculate the Jacobian determinant $|\nabla T_{\pm}|$ as well as the Banach indicatrix $\operatorname{Card}(T_{\pm}^{-1})$. While the former is carried out in \autoref{thm:jacobian}, the latter is quite a challenging task in general. However, in the remainder of this section we identify special cases for which we can determine $\operatorname{Card}(T_{\pm}^{-1})$.

\begin{lemma}\label{thm:jacobian}
For $\alpha \in C^1[0,L]$ and $\bn\in C^1([0,L],\sphere)$, the Jacobian determinant of $T_\pm$ is given by
\begin{equation}	\label{eq:D-simplified_1}
  \begin{aligned}
  \abs{\nabla T_{\pm}(k_1,k_2,t)}
    = \frac{k_0}{\kappa}
    {}& \left|
    \left( (1-\cos \alpha) \left(
    n_3\, 
    \bn'\cdot\bh
    -n_3' \bn\cdot\bh
    \right)
    -n_3 \, \bn\cdot\left( \bn' \times \bh \right) \sin \alpha 
    \right) \right.
    \\&
    \left.
    - \alpha'\,\left( n_1k_2-n_2k_1 \right) 
    + \left(\bn\cdot\bh\right) \left(n_1n_2'-n_2n_1'\right) \sin \alpha
    \right|
    ,
  \end{aligned}
\end{equation}
	where $\bh \coloneqq ( k_1,k_2, \pm \kappa - k_0)^\tT$. In particular $\abs{\nabla T_{\pm}} \in L^1(\mathcal{U})$.
\end{lemma}
The proof of this lemma is postponed until \autoref{sec:jacobian}.

For rotation axes independent of $t$ the Jacobian determinant of $T_\pm$ simplifies considerably. Moreover, the Banach indicatrix is constant almost everywhere.

\begin{corollary} \label{cor:detaxis}
	Let the assumptions of \autoref{thm:reconstruction} be satisfied. If $\bn'(t) = \bf0$, then
	\begin{equation} \label{eq:detaxis}
		\left\lvert \nabla T_\pm (k_1,k_2,t)\right\rvert = \frac{ k_0 \abs{\alpha'(t)}\abs{ n_2k_1 - n_1k_2 }}{\kappa}.
	\end{equation}
	If in addition $\bn \neq \be_3$ and $\alpha$ is strictly increasing with $\alpha(0) = 0$ and $\alpha(L)=2\pi$, then
	\begin{equation*}
	\operatorname{Card}(T_{\pm}^{-1}(T_{\pm}(k_1,k_2,t))) = 2
	\end{equation*}
	for almost every $(k_1,k_2,t) \in \mathcal{U}$.
\end{corollary}

\begin{proof}
	\autoref{eq:detaxis} is a direct consequence of \autoref{eq:D-simplified_1}. It remains to show the second statement.
  Let $(\hat k_1, \hat k_2, \hat t) \in \mathcal{U}$ be given.
  We want to find the number of points $(k_1,k_2,t) \in \mathcal{U}$ which satisfy
  \begin{equation}\label{eq:numsol}
  T_\pm (k_1,k_2,t) = T_\pm (\hat k_1, \hat k_2, \hat t).
  \end{equation}
 First, we point out that there can be no $t \neq \hat t$ such that 
  \begin{equation*}
  T_\pm (\hat k_1,\hat k_2,t) = T_\pm (\hat k_1, \hat k_2, \hat t),
  \end{equation*}
  unless $T_\pm (\hat k_1, \hat k_2, \hat t)$ lies on the rotation axis or $\hat t \in \{0,L\}$. These cases, however, correspond to a subset $\mathcal{V} \subset \mathcal{U}$ of measure zero.

Thus we look for $(k_1,k_2) \neq (\hat k_1,\hat k_2)$. Denote a general point on the hemisphere by $\bh^\pm = \bh^\pm (k_1,k_2) \coloneqq (k_1,k_2,\pm \kappa-k_0)^\tT$ and set $\hat{\bh}^\pm \coloneqq \bh^\pm (\hat k_1,\hat k_2)$. Since $T_\pm$ is the composition of $\bh^\pm$ with a rotation and rotations preserve the lengths of vectors, we deduce from \autoref{eq:numsol} that
\begin{equation}\label{eq:samenorm}
	\| \bh^\pm \| = \| \hat{\bh}^\pm \|.
\end{equation}
Calculating these norms shows that $\kappa(k_1,k_2) = \kappa(\hat k_1,\hat k_2)$ and that
  \begin{equation} \label{eq:k12-equal}
    k_1^2+k_2^2
    = \hat k_1^2+\hat k_2^2.
  \end{equation}
  Furthermore, the rotation $R_{\bn,\alpha}$ does not affect the vector component in the direction of the rotation axis $\bn$, that is,
	\begin{equation}\label{eq:sameplane}
  \bh^\pm \cdot \bn = R_{\bn, \alpha(t)} \bh^\pm \cdot \bn
  = R_{\bn, \alpha(\hat t)} \hat{\bh}^\pm \cdot \bn = \hat{\bh}^\pm \cdot \bn
	\end{equation}
  and consequently
  $$
  0
  =\bn\cdot(\bh^\pm - \hat\bh^\pm )
  = n_1(k_1-\hat k_1) + n_2(k_2-\hat k_2).
  $$
  Since, by assumption, $n_1^2+n_2^2\neq0$, there exists $\lambda\in\R$ such that
  $
  k_1-\hat k_1=-\lambda n_2
  $ and 
  $
  k_2-\hat k_2=\lambda n_1.
  $
  Hence,
  $(k_1, k_2) = (\hat k_1 - \lambda n_2, \hat k_2+ \lambda n_1)$. Now we conclude from
  \autoref{eq:k12-equal} that 
  $$
  0
  = \lambda^2(n_1^2+n_2^2) +2\lambda (\hat k_2 n_1-\hat k_1 n_2),
  $$
  which has the two solutions $\lambda_1=0$ and $\lambda_2 = 2(\hat k_2 n_1-\hat k_1 n_2)/(n_1^2 + n_2^2)$. The former corresponds to $(k_1, k_2) = (\hat k_1, \hat k_2)$ and the latter to a reflection of $(\hat k_1,\hat k_2)$ across the line passing through the origin with direction $(n_1,n_2)$. We can ignore the possibility that the two solutions coincide, as this corresponds to a set of measure zero.
  
It remains to count the $t\in [0,L]$ which satisfy
\begin{equation}\label{eq:samecircle}
T_\pm(\hat k_1 - \lambda_2 n_2, \hat k_2+ \lambda_2 n_1, t) = T_\pm (\hat k_1, \hat k_2, \hat t).
\end{equation}
\autoref{eq:samenorm} and \autoref{eq:sameplane} imply that during rotation $\hat \bh^\pm$ and $\bh^\pm(\hat k_1 - \lambda_2 n_2, \hat k_2+ \lambda_2 n_1) $ move along the same circle around the rotation axis. Since, by assumption, the range of $\alpha$ is $[0,2\pi]$, both points on the hemisphere make a full turn. So there must be at least one $t$ satisfying \autoref{eq:samecircle}. But since $\alpha$ is also bijective, there is exactly one such $t$. Thus we have shown that for almost every $(\hat k_1, \hat k_2, \hat t) \in \mathcal{U}$ there is exactly one other point $(k_1, k_2, t)$ such that \autoref{eq:numsol} holds. Therefore
 $\operatorname{Card}(T_{\pm}^{-1}(T_{\pm}(\hat k_1,\hat k_2,\hat t))) = 2$ almost everywhere in $\mathcal{U}$.
\end{proof}

Let us consider a simple example which is also treated in the numerical part.

\begin{example}[Full uniform rotation around the $r_1$-axis] \label{ex1}
	We consider rotation around the $r_1$-axis with rotation matrix
	\begin{equation*}
		R_{\be_1,\alpha}^\tT=
		\begin{pmatrix}
			1 & 0 & 0\\ 0 & \cos \alpha &-\sin \alpha \\ 0 &\sin \alpha  & \phantom{-} \cos \alpha 
		\end{pmatrix}
	\end{equation*}
	and $\alpha(t) = t$, $t \in [0,2\pi]$. Instead of $t$ we write $\alpha$ below.
From \autoref{cor:detaxis} it follows that
	\begin{equation} \label{eq:approx2}
		f^{\pm}_{\mathrm{bp}} (\br)=\frac{-\i k_0}{4\pi^2}
		\int_{\mathcal{U}}	\e^{\i (R_{\be_1,\alpha}\bh^{\pm}\cdot\br-\kappa r_M)}\ktran_{1,2}u_\alpha(k_1,k_2,\pm r_M)\abs{k_2}\, d(k_1,k_2,\alpha) \,,
	\end{equation}
  where $\bh^{\pm} \coloneqq (k_1,k_2,\pm\kappa-k_0)^\tT$.

We want to illustrate the corresponding sets $T_{\pm}(\mathcal U)$. 
Consider $\by\in T_{+}(\mathcal{U})$, i.e. $\by = (y_1,y_2,y_3)^\tT= R_{\be_1,\alpha} (k_1,k_2,\kappa-k_0)^\tT $ where $k_1^2+k_2^2<k_0^2$ and $0\leq\alpha\leq 2\pi$.
Then, it holds that $y_1=k_1$ and
\begin{align*}
y_2^2 + y_3^2
&= k_2^2 + (\kappa-k_0)^2
= k_2^2 + k_0^2- k_1^2-k_2^2 - 2\kappa k_0 + k_0^2
\\&
= 2k_0^2 - 2 k_0 \sqrt{k_0^2-k_1^2-k_2^2} - k_1^2.
\end{align*}
Since  $k_2^2 \in[0, k_0^2-k_1^2)$, this implies
\begin{align*}
y_2^2 + y_3^2 
&< 2k_0^2  - y_1^2,
\\
y_2^2+y_3^2
&\ge
2k_0^2 - 2 k_0 \sqrt{k_0^2-y_1^2} - y_1^2
= \left(k_0-\sqrt{k_0^2-y_1^2}\right)^2.
\end{align*}
The first equation can be rewritten as $\norm{\by}^2 < 2 k_0^2$, while
the second equation gives
\begin{align*}
\|(y_2,y_3)\| &\ge k_0 - \sqrt{k_0^2-y_1^2}\\
k_0 - \|(y_2,y_3)\| &\le \sqrt{k_0^2-y_1^2}.
\end{align*}
Noting that $k_0 - \|(y_2,y_3)\| > -  \sqrt{k_0^2-y_1^2}$ and
taking the square, we get $(\norm{(y_2,y_3)}-k_0)^2 + y_1^2 \le k_0^2$.
In summary, we see that
\begin{equation} \label{eq:Aaxis}
T_+(\mathcal{U})
= \{ \by\in\R^3 : \norm{\by}^2 < 2 k_0^2,\, (\norm{(y_2,y_3)}-k_0)^2 + y_1^2 \le k_0^2 \},
\end{equation}
which is displayed in \autoref{fig:cover2}. Similar considerations for $T_-(\mathcal{U})$ show that this set together with $T_+(\mathcal{U})$ generates a solid horn torus of radius $k_0$, see \autoref{fig:cover2} right. Finally we point out that a rotation around any other axis in the $r_1$-$r_2$-plane leads to a k-space coverage $T_\pm (\mathcal{U})$ which is a rotated version of the one displayed in \autoref{fig:cover2}.
\end{example}

\begin{figure}[h!]
  \centering
  \includegraphics[width=0.35\textwidth]{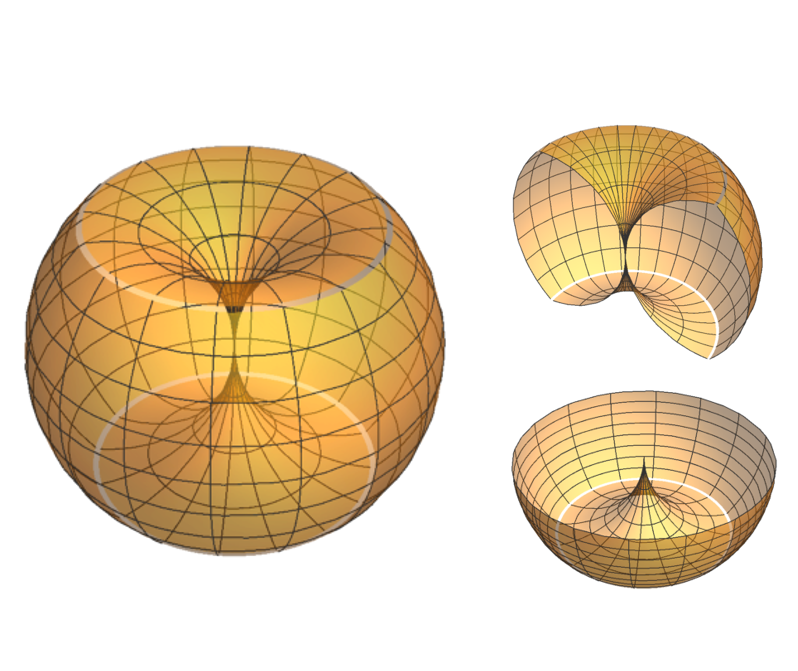} \qquad
  \begin{tikzpicture}[scale=0.7]
    \draw[->]   (0,-4) -- (0,4) node[below left] {$k_1$};
    \draw (0,0) circle (2.12);
    \draw (0,0) circle (3);
    \draw (1.5,1.5) arc (90:270:1.5);
    \draw (-1.5,-1.5) arc (-90:90:1.5);
    \draw (1.5,-1.5) arc (-90:90:1.5);
    \draw (-1.5,1.5) arc (90:270:1.5);
    \fill[red] (1.5,1.5) arc (90:270:1.5) -- (1.5,-1.5) arc (-45:45:2.12);
    \fill[red] (-1.5,-1.5) arc (-90:90:1.5) -- (-1.5,1.5) arc (135:225:2.12);
    \draw[->]   (0,0) -- (-1.5,1.5);
    \node[below] at (-1,0.8) {$\sqrt{2}k_0$};
    \draw[->]   (1.5,0) -- (0.5,-1.12);
    \node[below] at (1.17,-0.6) {$k_0$};
    \draw[->]   (0,0) -- (1,2.83);
    \node[below] at (1.2,2.5) {$2k_0$};
  \end{tikzpicture} \qquad 
  \begin{tikzpicture}[scale=0.7]
    \draw[->]   (0,-4) -- (0,4) node[below left] {$k_1$};
    \draw (0,0) circle (2.12);
    \draw (0,0) circle (3);
    \draw (1.5,1.5) arc (90:270:1.5);
    \draw (-1.5,-1.5) arc (-90:90:1.5);
    \draw (1.5,-1.5) arc (-90:90:1.5);
    \draw (-1.5,1.5) arc (90:270:1.5);
    \fill[red] (1.5,1.5) arc (45:-45:2.12) -- (1.5,-1.5) arc (-90:90:1.5);
    \fill[red] (-1.5,-1.5) arc (225:135:2.12) -- (-1.5,1.5) arc (90:270:1.5);
    \draw[->]   (0,0) -- (-1.5,1.5);
    \node[below] at (-1,0.8) {$\sqrt{2}k_0$};
    \draw[->]   (1.5,0) -- (0.5,-1.12);
    \node[below] at (1.17,-0.6) {$k_0$};
    \draw[->]   (0,0) -- (1,2.83);
    \node[below] at (1.2,2.5) {$2k_0$};
  \end{tikzpicture}
  \caption{Frequency coverage $T_{\pm}(\mathcal{U})$ for a full rotation about the $r_1$-axis. 
    \emph{Left:} 3D visualization of $T_+(\mathcal{U})$ (transmission imaging).
    \emph{Middle:} Cross section of $T_+(\mathcal{U})$.
    \emph{Right:} Cross section of $T_-(\mathcal{U})$ (reflection imaging).}
  \label{fig:cover2}
\end{figure}
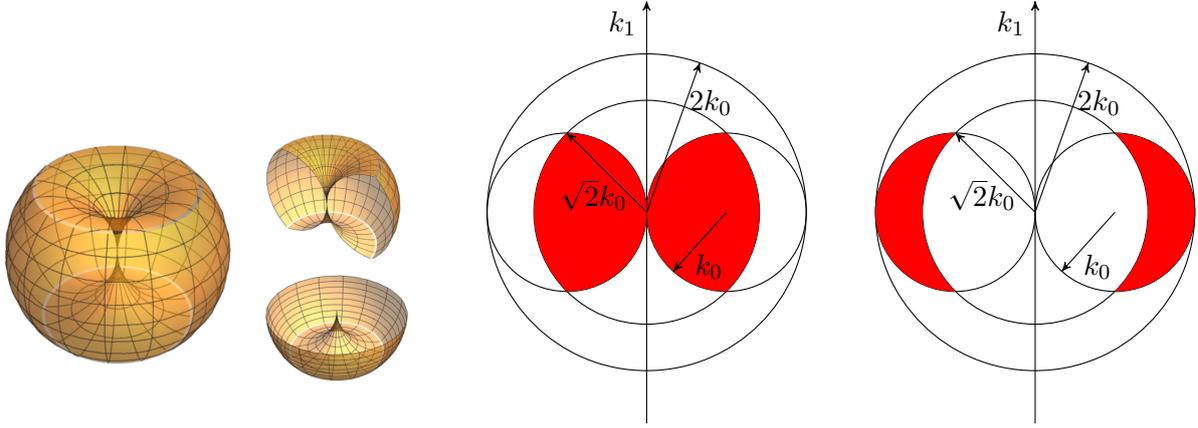

\begin{remark}[Maximal and minimal k-space coverage]
Subsequent rotations about two orthogonal axes in the $r_1$-$r_2$-plane generate a radially symmetric k-space coverage according to \autoref{fig:cover3}. This is the maximal coverage that can be obtained within the experimental setup. In contrast, the rotation around the $r_3$-axis does not provide additional information in k-space, recall \autoref{fig:coverage}.
\end{remark}

\begin{figure}[h!]
	\centering	
    \begin{tikzpicture}[scale=0.7]
	\fill[red] (0,0) circle (2.12);
	\draw[->]   (0,0) -- (-1.5,1.5);
	\node[below] at (-1,0.8) {$\sqrt{2}k_0$};
	\draw[->]   (0,0) -- (1,2.83);
	\node[below] at (1.2,2.5) {$2k_0$};
	\draw[->]   (0,-4) -- (0,4) node[below left] {$k_1$};
	\draw (0,0) circle (2.12);
	\draw (0,0) circle (3);
	\end{tikzpicture} \qquad 
	\begin{tikzpicture}[scale=0.7]
	\fill[red] (0,0) circle (3);
	\fill[white] (0,0) circle (2.12);
	\draw[->]   (0,0) -- (-1.5,1.5);
	\node[below] at (-1,0.8) {$\sqrt{2}k_0$};
	\draw[->]   (0,0) -- (1,2.83);
	\node[below] at (1.2,2.5) {$2k_0$};
	\draw[->]   (0,-4) -- (0,4) node[below left] {$k_1$};
	\draw (0,0) circle (2.12);
	\draw (0,0) circle (3);
	\end{tikzpicture}
	\caption{Frequency coverage for subsequent rotations around two orthogonal axes in the $r_1$-$r_2$ plane. 
	\emph{Left:} Cross section of transmission imaging. \emph{Right:} Cross section of reflection imaging.}
	\label{fig:cover3}
\end{figure}
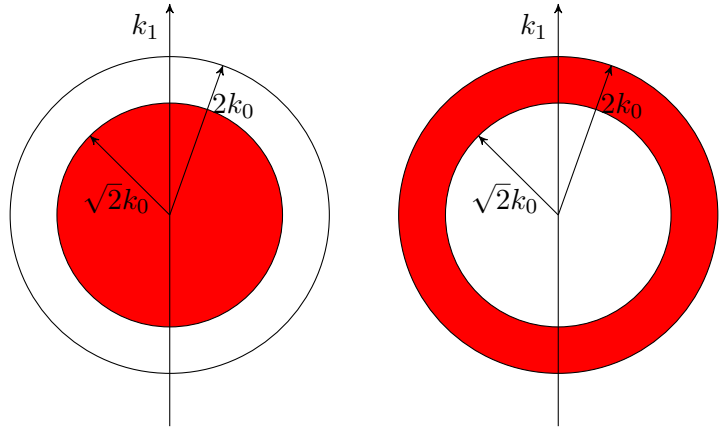

In the previous examples, the rotation axis $\bn$ was kept constant.
In the following, we consider a moving rotation axis $\bn(t)$, $t \in [0,2\pi]$, in the setup of transmission imaging.

\begin{proposition} \label{prop:mov-axis}
 Let $\bn \in C^1(([0,2\pi], \sphere)$ and $\alpha\in C^1[0,L]$.
 We denote the trajectory of $\be_3=(0,0,1)^\tT$ under $T_+$ by
  $$
  \be(t)
  \coloneqq
  R_{\bn(t),\alpha(t)} \be_3,
  \quad t\in[0,L].
  $$
  Let $\by\in\R^3$ with $\norm{\by}< \sqrt2k_0$.
  Then, there exists $(k_1,k_2,t)\in\mathcal{U}$ with $T_+(k_1,k_2,t)=\by$
  if and only if
  \begin{equation} \label{eq:ye}
    \by\cdot \be(t)
    = -\frac{\norm{\by}^2}{2k_0}.
  \end{equation}
  Moreover,
  $\operatorname{Card}(T_+^{-1}(\by))$ 
	equals the number of solutions $t\in[0,L]$ of \autoref{eq:ye}.
\end{proposition}

\begin{proof}
  Assume that there exists $(k_1,k_2,t)\in\mathcal{U}$ such that
  $\by = T_+(k_1,k_2,t) = R_{\bn(t),\alpha(t)} \bh$ with
  $\bh = (k_1,k_2,\kappa-k_0)$.
  Since the norm is invariant under rotations, it holds that
  \begin{equation} \label{eq:opt}
    \norm{\by}^2
    = \norm{\bh}^2
    = 2k_0(k_0-\kappa),
  \end{equation}
  which implies that $\kappa = \sqrt{k_0^2-k_1^2-k_2^2}$ depends only on $\norm{\by}$
  and, by the invariance of the scalar product under the rotation $R_{\bn(t),-\alpha(t)}$, it follows from 
  \autoref{eq:opt} that
  $$
  \by\cdot \be(t)
  =
  \bh\cdot \be_3
  = h_3
  = \kappa-k_0
  = -\frac{\norm{\by}^2}{2k_0},
  $$
  which shows \autoref{eq:ye}.
  
  Conversely, let $\by\in\R^3$ with $\norm{\by}< \sqrt2k_0$, and let $t$ satisfy \autoref{eq:ye}.
  Set 
  $
  \bk = (k_1,k_2,k_3)^\tT \coloneqq R_{\bn(t),-t}\by.
  $
  We show that
  $(k_1,k_2,t)\in{\mathcal{U}}$
  and
  $T_+(k_1,k_2,t)=\by$, i.e, that
  $
  k_3 = \sqrt{k_0^2-k_1^2-k_2^2} - k_0.
  $
  We have by \autoref{eq:ye} and the invariance of the scalar product and vector lengths with respect to rotations
  $$
   k_3
  =\bk \cdot \be_3
  = \by \cdot \be(t)
  = \frac{-\norm{\by}^2}{2k_0} = \frac{-\norm{\bk}^2}{2k_0}
  > - k_0
  $$
	so that $k_3+k_0 > 0$.
	Moreover, the above equation implies that
	\begin{equation} \label{eq:opt2}
	(k_3 + k_0)^2 = -k_1^2 - k_2^2 + k_0^2.
	\end{equation}
	Thus, $k_1^2+k_2^2 < k_0^2$, 
	so that $(k_1,k_2,t)\in\mathcal U$.
	Taking  the square root in \autoref{eq:opt2}, we obtain the desired form of $k_3$.
	Since, for fixed $t$, the map
  $(\tilde k_1,\tilde k_2)\mapsto T_+(\tilde k_1,\tilde k_2,t)$
  is one-to-one, 
  we have shown the assertion.
\end{proof}

\begin{example}[Half rotation around the $r_1$-axis] \label{ex:half-rotation}
  We consider the same fixed rotation axis as in \autoref{ex1},
  but we restrict the angle $\alpha(t)=t$, $t\in[0,\pi]$.
  In what follows, we show that the Banach indicatrix $\operatorname{Card}(T_+^{-1}(\cdot))$ takes different values on sets of positive measure, see \autoref{fig:coverage-half}.
  
  Let $\by\in T_+(\mathcal{U}) \setminus\{\mathbf0\}$.
  Since $T_+(\mathcal{U})$ is a subset of the k-space coverage for the full rotation, \autoref{eq:Aaxis} shows that $\norm{\by}^2 < 2 k_0^2$ and $(\norm{(y_2,y_3)}-k_0)^2 + y_1^2 \le k_0^2$, which yields $\norm{\by}^2\le 2k_0 \sqrt{y_2^2+y_3^2}$.
  By \autoref{prop:mov-axis}, the Banach indicatrix $\operatorname{Card}(T_+^{-1}(\by))$ is equal to the number of solutions $t\in[0,\pi]$ of 
  \begin{equation}\label{eq:ye-h}
    y_2 \sin t + y_3 \cos t 
    = -\frac{\norm{\by}^2}{2k_0}.
  \end{equation}
  Let us denote the left side by
  $$\psi(t)\coloneqq y_2 \sin t + y_3 \cos t,\quad t\in[0,\pi].$$
  If $y_2,y_3\neq0$, the derivative $\psi'(t) = y_2\cos t - y_3\sin t$ has only one zero $t_0=\arctan\frac{y_2}{y_3}$ in $[0,\pi]$, where we use the branch of the arctangent with range $[0,\pi)$.
  If $y_3=0$ and $y_2\neq0$, then $\psi'$ has the unique zero $t_0=\frac\pi2$, which comes along with setting $\arctan \frac{y_2}{0}=\frac\pi2$.
  If $y_2=0$ and $y_3\neq0$, then $\psi$ is monotone on $[0,\pi]$ and thus \autoref{eq:ye-h} has one solution.
  If $y_2=y_3=0$, also $y_1=0$ and thus there are infinitely many solutions.
  
  In the case $y_2<0$,
  we see that the unique zero $t_0\in[0,\pi]$ of $\psi'$ fulfills ${\sin t_0}= y_2 {\sqrt{y_2^2+y_3^2}}$
  and ${\cos t_0}= y_3{\sqrt{y_2^2+y_3^2}}$.
  Then we obtain $\psi(t_0) = -\sqrt{y_2^2+y_3^2}<0$.
  The second derivative is $\psi''(t) = -\psi(t)$ and hence $\psi''(t_0)>0$.
  Together with the continuity of $\psi'$, this implies that $\psi$ decreases on $[0,t_0]$ and increases on $[t_0,\pi]$.
  We have the two local maxima $\psi(0)=y_3$ and $\psi(\pi)=-y_3$ and the minimum $\psi(t_0)=-\sqrt{y_2^2+y_3^2}$.
  Hence, \autoref{eq:ye-h} has two solutions if $-\frac{\norm{\by}^2}{2k_0} \le -\abs{y_3}$ and otherwise one solution if $0>-\frac{\norm{\by}^2}{2k_0} >- \abs{y_3}$.
  Note that we have already seen that ${\norm{\by}^2}
  \le {2k_0} \sqrt{y_2^2+y_3^2}$.
  
  An analogous consideration for $y_2>0$ shows that $\phi$ has its maximum $\sqrt{y_2^2+y_3^2}$ at $t_0$ and goes monotonically to the minima at the boundary $\psi(0)=y_3$ and $\psi(\pi)=-y_3$.
  Then \autoref{eq:ye-h} has one solution if $\frac{-\norm{\by}^2}{2k_0}\ge -\abs{y_3}$ and no solution otherwise.
  In conclusion, we have for $0<\norm{\by}^2< {2} k_0^2$,
  \begin{equation*}
    \operatorname{Card}(T_+^{-1}(\by))
    =\begin{cases}
      2 ,& y_2 < 0 $ and $ 2k_0\abs{y_3} \le \norm{\by}^2 \le 2k_0\sqrt{y_2^2+y_3^2},\\
      1 ,& y_2 < 0 $ and $ 2k_0\abs{y_3} > \norm{\by}^2,\\
      1 ,& y_2 > 0 $ and $ 2k_0\abs{y_3} \ge \norm{\by}^2,\\
      1 ,& y_2 = 0 ,\\
      0 ,& $otherwise$.
    \end{cases}
  \end{equation*}
  \begin{figure}\centering
    \begin{tikzpicture}[scale=0.66]
      \draw[gray,->,dashed]   (-3,0) -- (3,0) node[below,black] {$y_2$};
      \draw[gray,->,dashed]   (0,-2.5) -- (0,3) node[below left,black] {$y_3$};
      \fill[red,opacity=.5] ( 1.5,1.5) arc ( 0:-180:1.5) -- (1.5,1.5) arc (45:135:2.12);
      \fill[red,opacity=.5] (-1.5,-1.5) arc (180:0:1.5) -- (-1.5,-1.5) arc (225:315:2.12);
      \draw[black] (0,0) circle (2.12);
      \node[blue] at (-1.6,0.35) {$2$};
      \node[black] at ( 1.6,0.35) {$0$};
      \node[black] at (-0.25, 1.4) {$1$};
      \node[black] at (-0.25,-1.4) {$1$};
    \end{tikzpicture}
    \caption{Areas of constant Banach indicatrix $\operatorname{Card}( T^{-1}(\by))$ for the half rotation, sectional plot at $y_1=0$. \label{fig:coverage-half} }
  \end{figure}
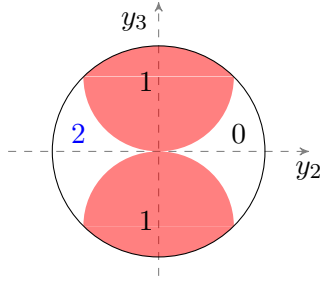
\end{example}

\begin{example}[Moving rotation axis] \label{ex:mov-axis}
  For $c>0$ arbitrary fixed, we consider the moving rotation axes
  $$
  \bn(t) = (\cos(c\,\sin t), \sin(c\,\sin t), 0)^\tT,
  $$ 
  and the rotation angle $\alpha(t) = t\in[0,2\pi]$.
  The axes $\bn(t)$ oscillate slightly around $\be_1$ in the $r_1$-$r_2$ plane.
  Then
  $$
  \be(t)
  =
  (-\sin(c\,\sin t)\, \sin t,
  \cos(c\,\sin t)\, \sin t,
  \cos t )^\tT.
  $$
  We consider
  $
  \by = (y_1,0,0)^\tT \in T_+(\mathcal U).
  $
  Then \autoref{prop:mov-axis} implies that $\operatorname{Card}(T_+^{-1}(\by))$ is the number of solutions $t$ of
  \begin{equation} \label{eq:ye_ex}
  \by\cdot\be(t)
  = -y_1 \sin(c\,\sin t)\, \sin t
  = -\frac{y_1^2}{2k_0}.
  \end{equation}
	We show that \autoref{eq:ye_ex}
  has 4 solutions if $0< y_1 < 2k_0 \sin c$.
  Since the function 
  $
  (0,\frac\pi2)\ni t
  \mapsto \sin(c\,\sin t)\, \sin t
  $
  is the composition and product of the positive, strictly increasing sine function, it is
  strictly increasing with supremum $\sin c$.
  So there is one solution $\hat t \in (0,\frac\pi2)$ of \autoref{eq:ye_ex}.
  The other three solutions $\pi-\hat t$, $\pi+ \hat t$, and $2\pi- \hat t$ follow by the symmetry of the sine.
  Hence, we see that $\operatorname{Card}(T_+^{-1}(\by))=4$.
  Since $T_+(\mathcal U)\ni\by\mapsto\by\cdot\be(t)$ is continuous, this result holds also in a small neighborhood of $\by$.
  However, for the point $\by=(0,0,y_3)$, we see that \autoref{eq:ye} becomes
  \begin{equation}\label{eq:ye_ex2}
  -y_3 \cos t = -\frac{y_3^2}{2k_0}.
  \end{equation}
  If $0<y_3<2k_0$, \autoref{eq:ye_ex2} has 2 solutions $t\in(0,2\pi)$ and $\operatorname{Card}(T_+^{-1}(\by))=2$. By continuity, $\operatorname{Card}(T_+^{-1}(\cdot))=2$ then holds in a small neighborhood of $\by$.

  As $T_+$ maps sets of measure $0$ to sets of measure $0$, the considerations above show that there exist two subsets of $\mathcal{U}$, both of positive measure, on which $\operatorname{Card}(T_+^{-1}(T_+(\cdot)))$ attains different values. Note that this conclusion holds for any $c>0$ and therefore for very small movements of the rotation axis $\bn(t)$. Taking into account the non-triviality of determining the regions of $\mathcal{U}$ with different Banach indicatrix, this example hints to the complexity of applying the backpropagation formula \autoref{eq:reconstruction} for general rotations.
\end{example}

\section{Discrete backpropagation and inverse NDFT} \label{sec:disc}
Next, we are interested in the numerical reconstruction of the three-dimensional function $f$
from given two-dimensional functions $u_t$.
By \autoref{eq:recon_n}, we can build on the relation
$$
-\sqrt{\frac{2}{\pi}} \e^{-\i\kappa r_M} \kappa \i \ktran_{1,2} u_{t}(k_1,k_2,\pm r_M)
=\ktran f(T_\pm(k_1,k_2,t)), \quad (k_1,k_2,t) \in \mathcal U.
$$
In this paper, we work with simulated data.
We assume in the numerical part 
that the function $\mathcal F f$ is known on $\mathcal Y_\pm = T_\pm(\mathcal U)$.
More precisely, $\mathcal F f$ is either given analytically or
we approximate it by the computation described at the end of this section. 

For $N \in 2 \mathbb N$, let $\mathcal I_N \coloneqq \{-\frac{N}{2} + j: j=0,\ldots,N-1\}$.
We aim to reconstruct the function~$f$ with $\supp f\subset\mathcal{B}_{\rs}\subset\R^3$
at the equispaced $N \times N \times N$ grid 
$$
\mathcal R_N\coloneqq   \frac{2\rs}{N} \, \mathcal I_N^3 \subset [-\rs,\rs]^3.
$$
To this end, we assume that $\mathcal Ff$ is given on the following sampling set in the k-space:
let
\begin{equation} \label{eq:J}
	\mathcal U_{N,S} \coloneqq
	\left\{ (k_1,k_2,t) : (k_1,k_2)\in {2k_0}\mathcal I_N^2,\, k_1^2+k_2^2\le k_0^2,\, t = \frac{2\pi j}{S}, \, j=0,\ldots,S-1\right\}.
\end{equation}
In particular, for
$S \coloneqq \lceil\frac{4}{\pi}N\rceil$, 
the number of data points $M \coloneqq |\mathcal U_{N,\lceil\frac{4}{\pi}N\rceil}|$ 
is approximately equal to $N^3$. In this case, we 
use the abbreviation 
$\mathcal U_N\coloneqq\mathcal U_{N,\lceil\frac{4}{\pi}N\rceil}$.
Then the sampling points in the k-space are given by the $M$ points in 
\begin{equation} \label{eq:Y_N}
\mathcal Y_{N}^\pm \coloneqq T_\pm(\mathcal U_{N}).
\end{equation}

In the following, we propose two reconstruction techniques,
namely the discrete backpropagation and the inverse \emph{nonequispaced discrete Fourier transform} (NDFT).
Our numerical tests will indicate that the latter appears to be preferable.
%
\subsection{Discrete backpropagation}
The discrete backpropagation is directly based on a discretization of the integral in
\autoref{eq:approx2g} using the values of $\mathcal F f$ on $\mathcal Y_{N}^\pm$.
For $\br \in \mathcal R_N$, we approximate $f_{\mathrm{bp}}(\br)$ by
\begin{equation}	\label{eq:bp}
\mathbf f_{\mathrm{bp}}(\br)\coloneqq
(2\pi)^{-\frac32}	
\frac{ \pi L k_0^2}{N^3} 
\sum_{(k_1,k_2,t)\in \mathcal U_N} \ktran f(T_\pm(k_1,k_2,t))\, 
\mathrm e^{\i\br\cdot T_\pm(k_1,k_2,t)}\, \frac{\abs{\nabla T_\pm(k_1,k_2,t)}}{{\operatorname{Card}(T_\pm^{-1}(T_\pm(k_1,k_2,t)))}}.
\end{equation}
We will see that evaluating $\mathbf f_{\mathrm{bp}}$ on $\mathcal R_N$
is, up to the multiplicative constant in front of the sum, an adjoint NDFT applied to
$$\ktran f(T_\pm(k_1,k_2,t)) \frac{\abs{\nabla T_\pm(k_1,k_2,t)}} {\operatorname{Card}(T_\pm^{-1}(T_\pm(k_1,k_2,t)))}
,\quad (k_1,k_2,t) \in \mathcal U_N.$$

\begin{remark}
For the discrete backpropagation \autoref{eq:bp}, 
it is crucial to know the Banach indicatrix $\operatorname{Card}(T_\pm^{-1}(T_\pm(\cdot)))$,
which we computed in some special cases, see Examples \ref{ex1} and \ref{ex:half-rotation}.
However, even for a small movement of the rotation axis as in \autoref{ex:mov-axis}, it seems to be quite difficult to determine the Banach indicatrix in general.
Such considerations are not necessary when applying the inverse NDFT considered next.
\end{remark}

\subsection{Inverse NDFT}
To explain the inverse NDFT, we recall the NDFT first.
The NDFT is the linear operator 
$\mathbf F_N: \R^{N^3} \rightarrow \R^M$ 
defined for our vectors 
$\mathbf f_N  
\coloneqq
\left( f ( \br ) \right)_{\mathbf r \in \mathcal R_N}
=
\left( f \left(\frac{2\rs}{N} {\mathbf j} \right) \right)_{\mathbf j \in \mathcal I_N^3}$ 
elementwise by
\begin{equation}\label{eq:ndft}
	\mathbf F_N \mathbf f_N (\by) 
	\coloneqq \frac{8\rs^3}{N^3} \sum_{\br \in \mathcal R_N} f(\br) \e^{\i \br \cdot \by}  
	= \frac{8\rs^3}{N^3} \sum_{\mathbf j \in \mathcal I_N^3} f \left(\frac{2\rs}{N} {\mathbf j} \right) \e^{\i \frac{2\rs}{N}  \mathbf j \cdot \by }
	, \quad \by \in \mathcal Y_{N}^\pm,
\end{equation}
see \cite[Section 7.1]{PlPoStTa18}. In other words, we can consider 
the NDFT via the tensor 
$\mathbf F_N = \left(\e^{\i \frac{2\rs}{N}  \mathbf j \cdot \by} \right)_{\mathbf j \in \mathcal I_N^3, \by \in \mathcal Y_N^\pm}$.
Furthermore, it provides an approximation of the Fourier transform 
$$\ktran f(\by) \approx \mathbf F_N \mathbf f_N(\by), \qquad \by\in\mathcal Y_N^\pm.$$
Then, as already mentioned above, the discrete backpropagation formula \autoref{eq:bp}
is just the application of the adjoint NDFT to weighted values
$\left( \mathcal F f(\by) \right)_{\by \in \mathcal Y_N^\pm}$.
In contrast, the inverse NDFT reconstructs  the values $\left( f ( \br ) \right)_{\mathbf r \in \mathcal R_N}$
by solving the least squares problem
\begin{equation} \label{eq:indft}
\mathrm{argmin}_{\mathbf f \in \mathbb R^{N^3}} \| \mathbf F_N \mathbf f - (\mathcal F f(\by))_{\by \in \mathcal Y _N^\pm}\|_2^2.
\end{equation}
More precisely, we call a solution of this problem \emph{inverse NDFT} of $ (\mathcal F f(\by))_{\by \in \mathcal Y _N^\pm}$, 
see  \cite[Section 7.6.2]{PlPoStTa18}.

\begin{remark}[Fast Computation by NFFT]
Computationally we will solve the least squares problem \autoref{eq:indft} by a 
\emph{conjugate gradient method on the normal equations} (CGNE)
as proposed in \cite{kupo04}.
Each iteration step of the CGNE algorithm requires the computation of an NDFT and an adjoint NDFT.
Both the computation of the NDFT and its adjoint can be realized in an efficient way by
the so-called \emph{nonequispaced fast Fourier transform} (NFFT). 
The NFFT requires only $\mathcal O(N^3 \log N)$ arithmetic operations instead of $\mathcal O(N^6)$ operations for the NDFT
and is highly recommendable in 3D.
Depending on the choice of the inner parameters, the NFFT provides an arbitrarily tight approximation of the NDFT, see, e.g., \cite{bey95,duro93,st97}.
\end{remark}

\textbf{Computation of k-space data.}
If the function $\mathcal F f$ or its values at $\mathcal Y_N^\pm$ are not given, we have to synthesize them
for our numerical tests.
Since the function $f$ is known in our synthetic examples and can therefore be sampled on an arbitrary fine grid,
we choose $n \gg N$ and approximate the values $\mathcal F f(\by)$ 
by applying the
NDFT on the fine grid, i.e., we take
\begin{equation} \label{simu}
	\mathcal F f(\by) \approx
	\frac{8\rs^3}{n^3}
	\sum_{\br \in \mathcal R_n} f(\br) \e^{\i \br \cdot \by}  
	= \frac{8\rs^3}{n^3} \sum_{\mathbf j \in \mathcal I_n^3} f \left(\frac{2\rs}{n} {\mathbf j} \right) \e^{\i \frac{2\rs}{n}  \mathbf j \cdot \by }
	, \quad \by \in \mathcal Y_{N}
\end{equation}
as given values for both the discrete backpropagation and the inverse NDFT.
We note that \autoref{simu} resembles $\mathbf F_{n}$ evaluated on a different grid than in \autoref{eq:ndft}.

\section{Numerical tests} \label{sec:num_results}
In this section, we demonstrate the performance of the discrete backpropagation in \autoref{eq:bp} and the inverse NDFT defined in \autoref{eq:indft} by numerical examples.
For computing the NDFT and its adjoint, we apply
the NFFT software library~\cite{nfft3}. If not stated otherwise, we use the NFFT for all reconstructions.
We fixed the number of CGNE iteration steps in the inverse NDFT to 20.
In this section, we concentrate on the transmission imaging associated with $T_+$.

\subsection{Sampling}

The wavelength $\lambda$ is related with the wave number via $k_0 = \frac{2\pi}{\lambda}$.
Since the discretized Fourier transform $\mathbf{F}_Nf(\by)$ is $\frac{\pi N}{\rs}$-periodic in $\by$
and the data points $\by\in\mathcal Y_N^+$ satisfy $\norm{\by}\le\sqrt2k_0$,
it is reasonable that the model parameters satisfy $\sqrt2k_0 \le \frac{\pi N}{2\rs}$ or, equivalently, 
$$N\ge\frac{2\sqrt2 k_0\rs}{\pi} = \frac{4\sqrt2 \rs}{\lambda}.$$
In particular, we choose $N=\frac{2\sqrt2 k_0\rs}{\pi}$, where we note that $\rs$ can always 
be made larger in order to make sure that $N$ is an integer.
The distance between adjacent grid points $\br\in\mathcal R_N$ is then $\frac{2\rs}{N}=\frac{\lambda}{2\sqrt{2}}$, which only depends on the wavelength $\lambda$.
In our numerical tests, we fix the wavelength $\lambda=1$ such that all measurements in $\br$ are in multiples of the wavelength.
Hence, the wave number is $k_0=2\pi$.
The data points $T_+(k_1,k_2,t)$ in the k-space are on the grid $\mathcal Y_{N}^+$ described in the \autoref{eq:Y_N}.
Therefore all data points in $\mathcal Y_{N}^+$ are contained in a ball of radius $\sqrt2 k_0 \approx 8.89$.

In our first numerical tests, we choose
the grid size
$N=80$ 
which corresponds to the radius $\rs= \frac{\lambda N}{4\sqrt2}\approx 14.1\lambda$ of the maximal support of $f$.
Then we have $N^3=512\,000$ grid points in $\mathcal R_N$ and 496\,944 data points in $\mathcal Y_N^+$.
If not available analytically, we simulated the values $\mathcal F f(\by)$, $\by \in \mathcal Y_N^+$,
by an NDFT of length $n = 5N$ as in \autoref{simu}. 

We compare the reconstruction quality based on the the \emph{structural similarity index measure} (SSIM) \cite{ssim}
and the
\emph{peak signal-to-noise ratio} (PSNR) determined by
$$
\operatorname{PSNR}(\mathbf f,\mathbf g)
\coloneqq 10 \log_{10} \frac{\max_{\br\in\mathcal R_N} \abs{\mathbf f(\br)}^2}{N^{-3} \sum_{\br\in\mathcal R_N} \abs{\mathbf f(\br) - \mathbf g(\br)}^2},
$$
where $\mathbf f$ is the ground truth and $\mathbf g$ is the reconstructed value. 
Note that higher values indicate a better reconstruction quality for both.
If $\mathcal F f(\by)$, $\by \in \mathcal Y_N^+$ is computed from the function values at a fine grid,
then it appears reasonable to take as ground truth for the PSNR the voxel values which are the averages 
of their five-point neighborhood on the fine grid
$$
\mathbf{f}^{\mathrm{av}}_N(\br)
\coloneqq
\frac{1}{5^3}
\sum_{\mathbf j\in\{-2,-1,\dots,2\}^3}  f \left(\br+ \frac{2\rs}{5N}\mathbf j\right) 
,\quad \br\in \mathcal R_N.
$$

\subsection{Function with exactly known Fourier transform}
In order to illustrate the effectiveness of the proposed reconstruction algorithms, 
we make the first test with the characteristic function
$f_a(\br) = \mathbf 1_{\mathcal{B}_{a}}(\br)$ of the ball with radius $a>0$.
Its Fourier transform is known analytically,
\begin{equation} \label{eq:Fcirc}
\ktran f_a(\by) = \sqrt{\frac2\pi}\frac1{\norm{\by}^3} \left( \sin\norm{a\by} - \norm{a\by}\cos\norm{a\by}\right),
\qquad\by\in\R^3,
\end{equation}
see \cite[p. 183]{PlPoStTa18}.
The object is fully rotated around the axis $\bn(t)=\be_1$ and angle $\alpha(t)=t$, $t\in[0,2\pi]$, as in \autoref{ex1}.
We consider the two cases that the input for our reconstruction arises from
i) the exact data $\ktran f_a (\by)$ by \autoref{eq:Fcirc}, 
and 
ii) the approximate data $\mathbf{F}_{5N} f_a$ by \autoref{simu}.
The error between the approximation $\mathbf{F}_{N}f_a(\by)$ and the true values $\ktran f_a (\by)$, $\by\in\mathcal Y_N^+$, is shown for different $N$ 
in \autoref{tab:err-forward}.

\begin{table}\centering
\begin{tabular}{lcccccc}
	\hline
	$N$ & 80 & 160 & 240 & 320 & 400 & 800
	\\\hline
	RMSE & 3.72\,E-2 & 9.37\,E-3 & 4.38\,E-3 & 2.58\,E-3 & 1.67\,E-3 & 4.46\,E-4
\end{tabular}
\caption{\label{tab:err-forward}
	The root mean square error (RMSE) of $\mathbf{F}_Nf_a$ with respect to $\ktran f_a$ in dependence on the grid size $N$ for the NDFT.}
\end{table}

The reconstruction is depicted in \autoref{fig:rec-exact}.
For both cases, we compare the inverse NDFT with the discrete backpropagation \autoref{eq:bp}. 
The latter shows stronger artifacts due to the sharp cutoff in the k-space.
There is almost no difference between the exact data in i) and approximate ones in ii);
since the approximate data is computed on a very fine grid.
\begin{figure}
	\begin{subfigure}[t]{.32\textwidth} 
		\includegraphics[width=\textwidth,trim={44 30 36 23},clip]{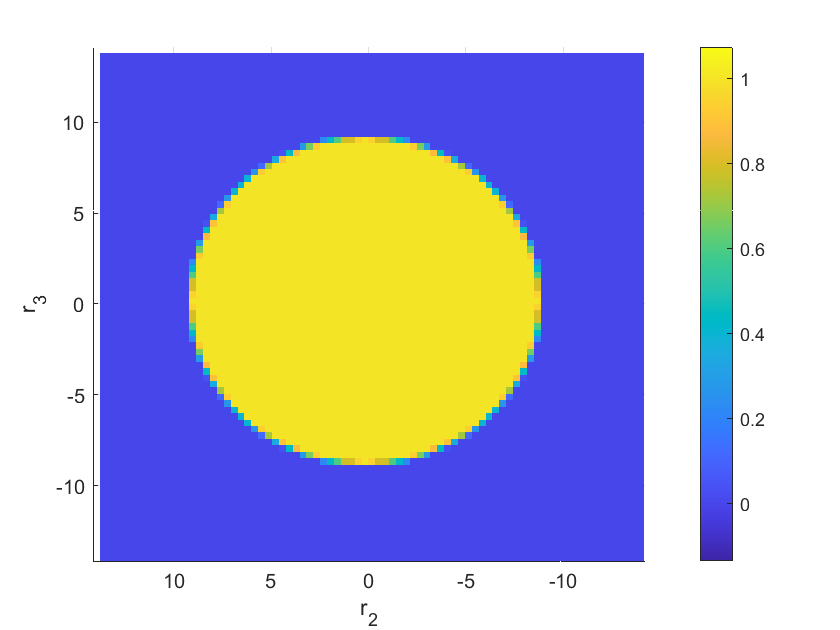}
		\caption{Test function discretized by $\mathbf f_{80}^{\mathrm{av}}$ on an $80\times80\times80$ grid}
	\end{subfigure}\hfill
	\begin{subfigure}[t]{.32\textwidth}
		\includegraphics[width=\textwidth,trim={44 30 36 23},clip]{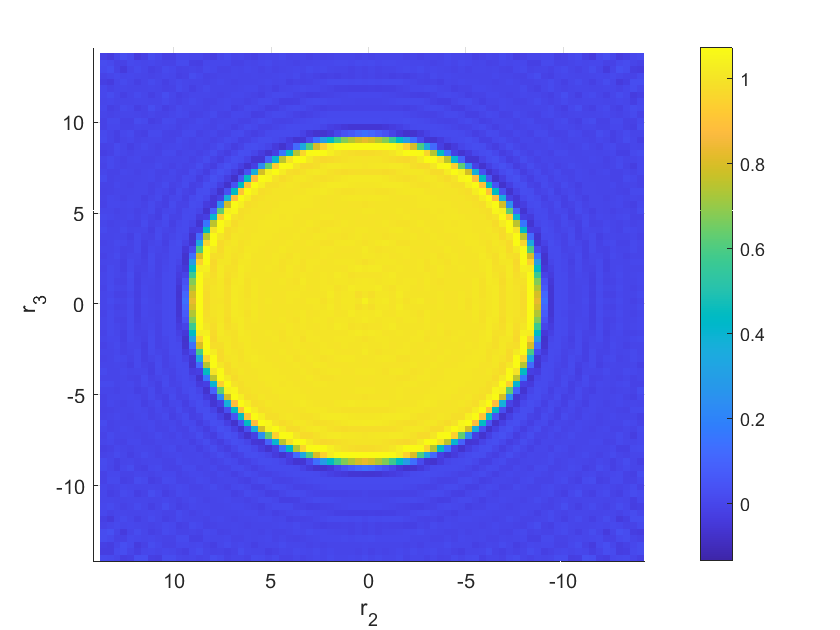}
		\caption{Inverse NDFT from approximate data
		\\     PSNR 32.61, SSIM 0.885}
	\end{subfigure}\hfill
	\begin{subfigure}[t]{.32\textwidth}
		\includegraphics[width=\textwidth,trim={44 30 36 23},clip]{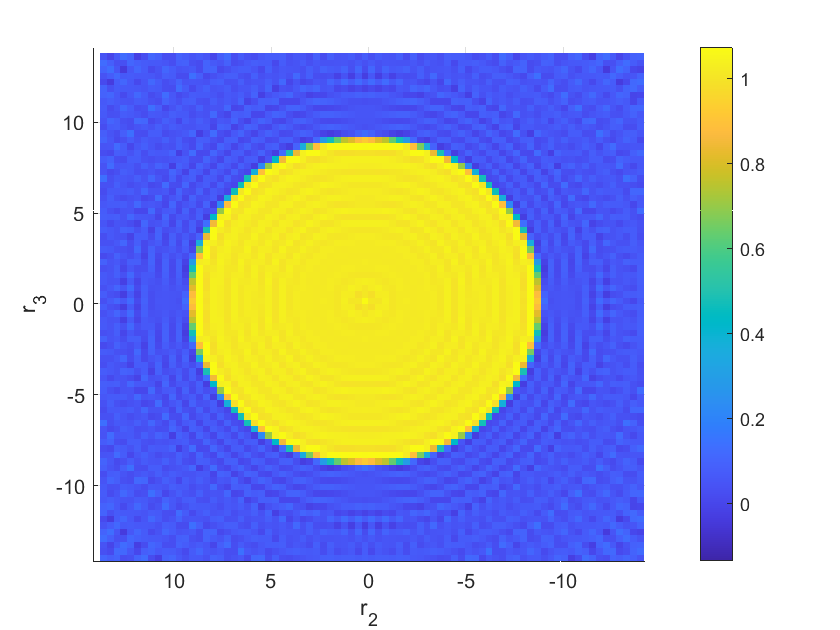}
		\caption{Backpropagation from approximate data
			\\     PSNR 27.02, SSIM 0.370}
	\end{subfigure}\\[\medskipamount]
\centering
	\begin{subfigure}[t]{.32\textwidth}
		\includegraphics[width=\textwidth,trim={44 30 36 23},clip]{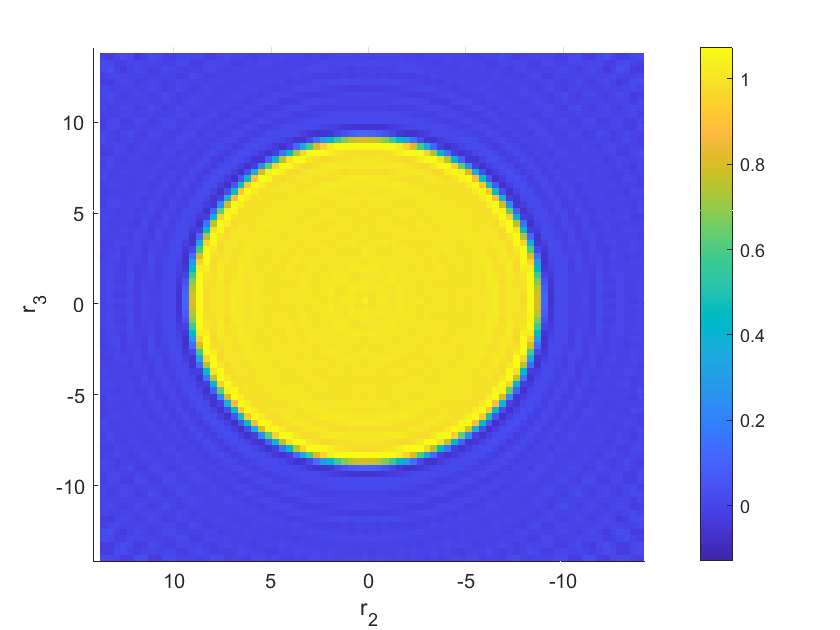}
	\caption{Inverse NDFT from exact data
	\\    PSNR 32.60, SSIM 0.885}
	\end{subfigure}~~
	\begin{subfigure}[t]{.32\textwidth}
		\includegraphics[width=\textwidth,trim={44 30 36 23},clip]{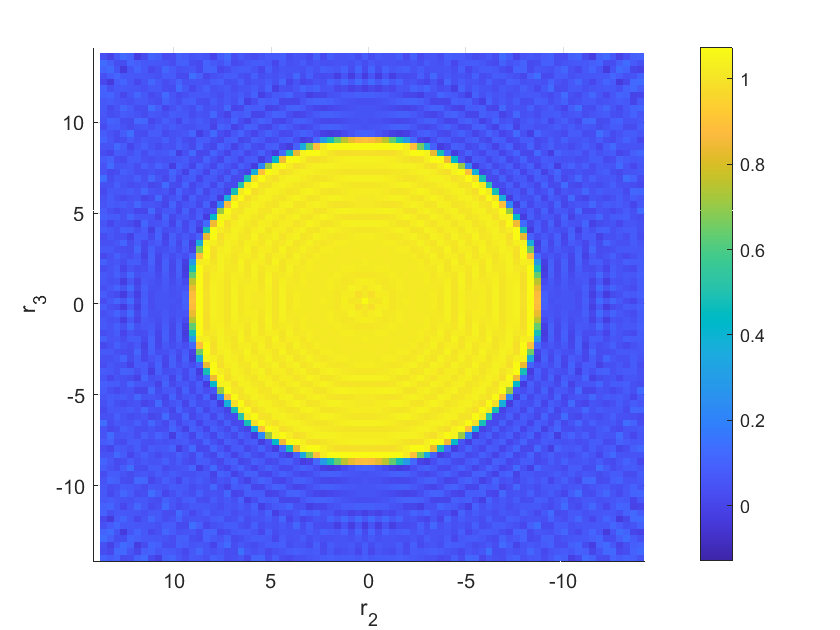}
	\caption{Backpropagation from exact data
	\\      PSNR 27.00, SSIM 0.370}
	\end{subfigure}
	\caption{Test with the characteristic function $f_a$ of a ball of radius $a=9$, with resolution $N=80$. 
		All plots show the slice of the 3D function at $r_1=-0.3$.
		The reconstructions use either the exact data $\ktran f_a$, or the approximated data from \autoref{simu}.
		In both cases, we use the same set of points $\mathcal Y_N^+$.
		\label{fig:rec-exact}}
\end{figure}

\subsection{Simple test function}
As a second test function $f$, we take the characteristic function of a ball and cut out a small segment around the plane $\{r_2=0\}$.
The rotation is around the $r_1$ axis
as in the previous test. The reconstruction results are shown in \autoref{fig:rec}, where the inverse NDFT yields an image 
with fewer artifacts than the discrete backpropagation.

\begin{figure}
  \centering
  \begin{subfigure}[t]{.32\textwidth} 
  	\includegraphics[width=\textwidth,trim={44 30 36 23},clip]{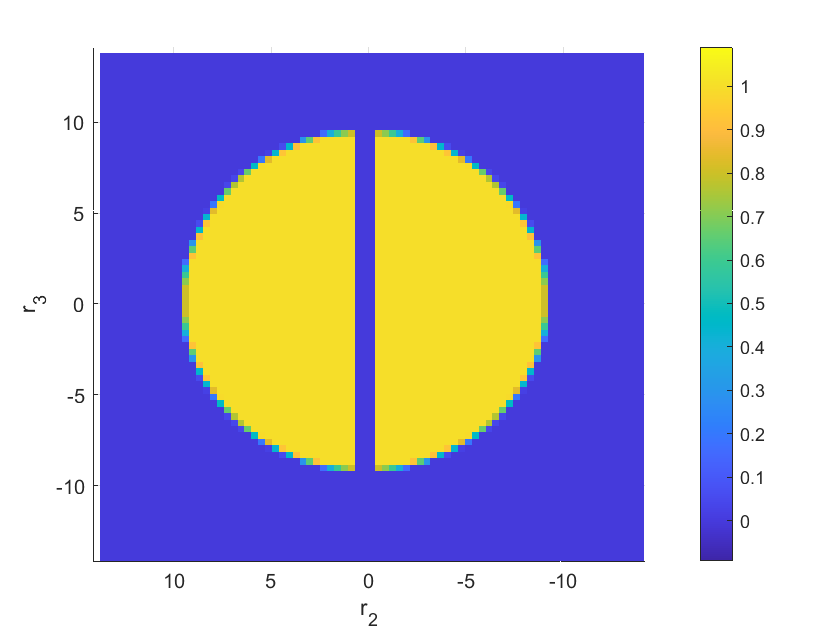}
  	\caption{Test function (ground truth image) on an $80\times80\times80$ grid}
  \end{subfigure}\hfill
	\begin{subfigure}[t]{.32\textwidth} 
		\includegraphics[width=\textwidth,trim={44 30 36 23},clip]{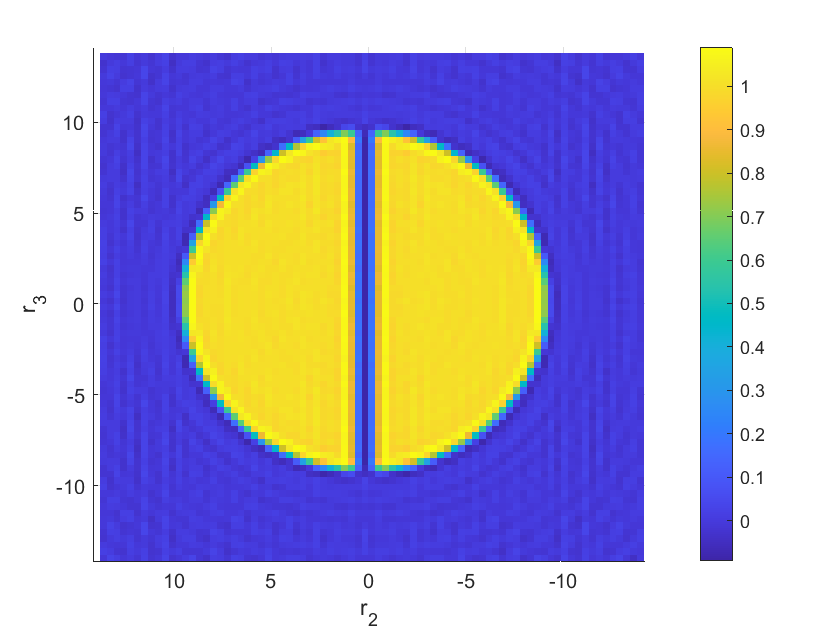}
		\caption{Inverse NDFT reconstruction
			\\     PSNR 29.52, SSIM 0.863}
	\end{subfigure}\hfill
	\begin{subfigure}[t]{.32\textwidth}
		\includegraphics[width=\textwidth,trim={44 30 36 23},clip]{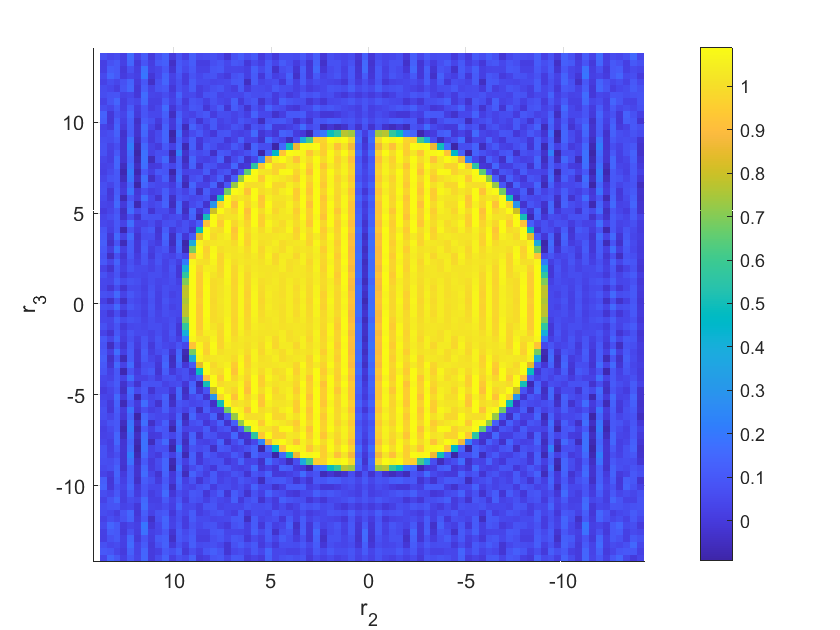}
		\caption{Backpropagation
			\\     PSNR 25.25, SSIM 0.366}
	\end{subfigure}\\[\bigskipamount]
	\begin{subfigure}[t]{.9\textwidth}
  \begin{center}
  \begin{tikzpicture}
    \begin{axis}[xlabel=$r_1$, width=.82\textwidth, height=0.42\textwidth, xmin=-14.1421, xmax=14.1421, enlarge y limits=0.05, axis x line*=bottom, axis y line=none, legend pos=outer north east, legend style={cells={anchor=west}}]
      \addplot+[thick,const plot mark mid,mark=none,black] table[x index=0,y index=1] {images/f1.dat};
      \addplot+[thick,const plot mark mid,mark=none,blue] table[x index=0,y index=2] {images/f1.dat};
      \addplot+[thick,const plot mark mid,mark=none,red,dotted,thick] table[x index=0,y index=3] {images/f1.dat};
      \legend{\small ground truth $f$,\small inverse NDFT,\small backpropagation};
    \end{axis}
  \end{tikzpicture}
  \end{center}
  \caption{Comparison of both reconstructions and the ground truth at the line $(r_1,0,0)^\tT$}
  \end{subfigure}
	\caption{\label{fig:rec}
    Top: Slice plot of the test function with discontinuity around the plane $r_2=0$, and its reconstructions.
    Bottom: Sectional plot of test function and its reconstructions.}
\end{figure}
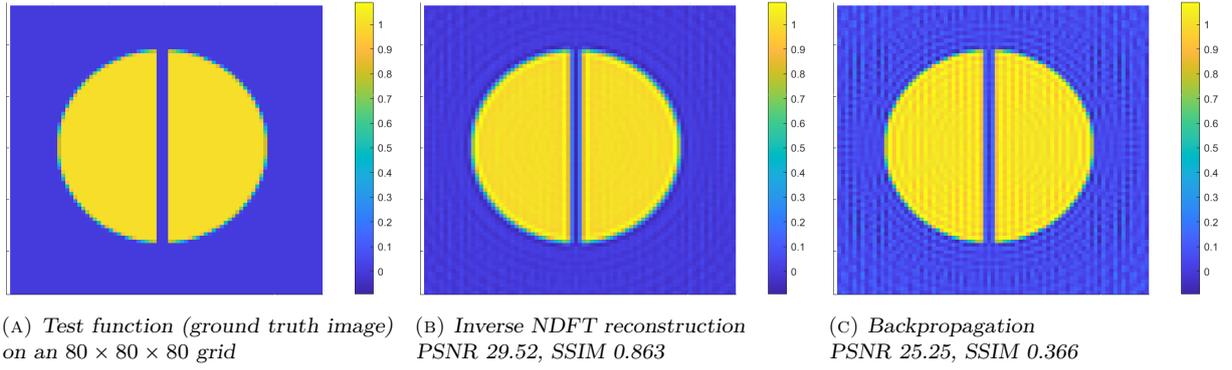
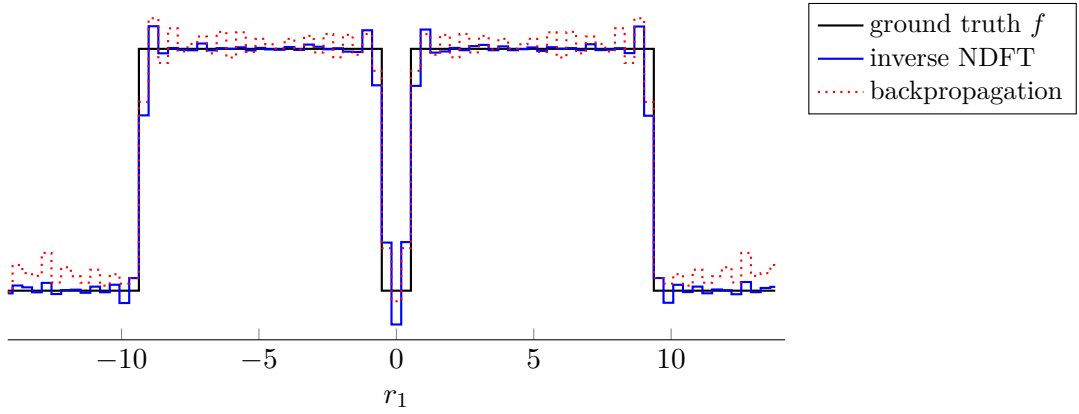

The reconstruction becomes more difficult when the object has discontinuities perpendicular to the rotation axis.
To this end, we take the previous test function, but the rotation is now around the $r_2$ axis.
The reconstruction in \autoref{fig:rec2}
resolves the gap considerably worse than in the previous example.
This is because we do not have any data of the Fourier transform $\ktran f$ near the $k_2$ axis away from the origin.
However, the information about the Fourier transform in this region is important 
due to the singularities of $f$ along planes perpendicular to the $r_2$ axis.
Again, the discrete backpropagation produces more artifacts than the inverse NDFT.

\begin{figure}
		\begin{subfigure}[t]{.48\textwidth}
		\includegraphics[width=\textwidth,trim={44 30 36 23},clip]{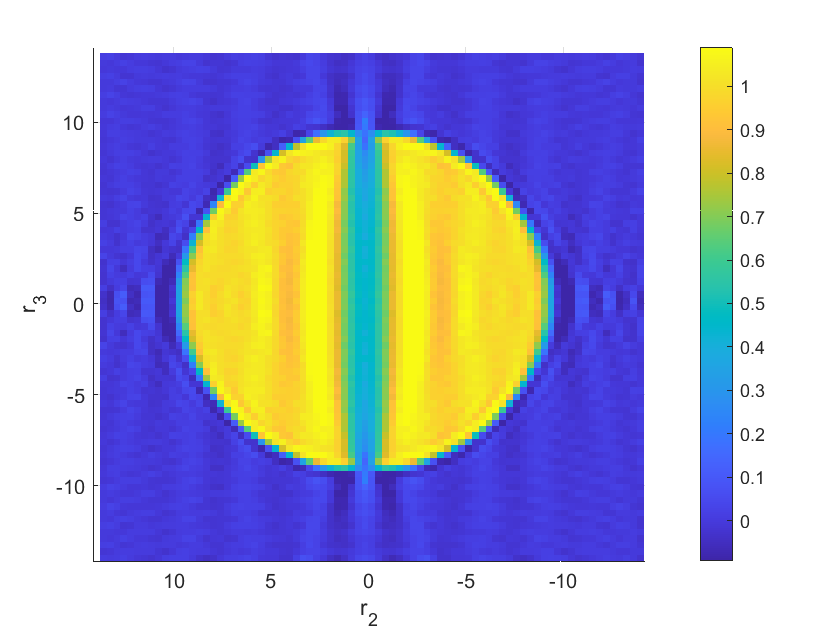}
		\caption{Inverse NDFT reconstruction
			\\     PSNR 22.94, SSIM 0.658}
	\end{subfigure}\hfill
		\begin{subfigure}[t]{.48\textwidth}
		\includegraphics[width=\textwidth,trim={44 30 36 23},clip]{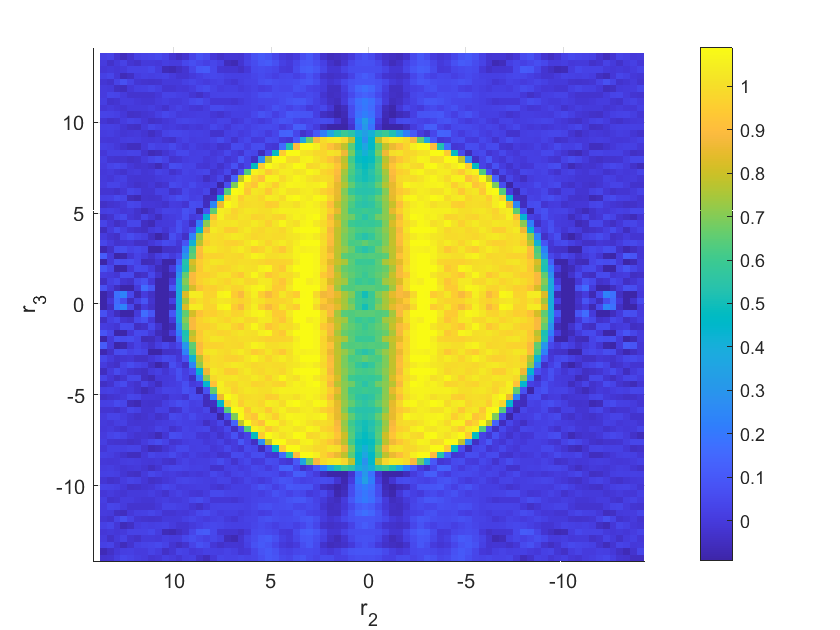}
		\caption{Backpropagation
			\\     PSNR 21.04, SSIM 0.341}
	\end{subfigure}
	\caption{\label{fig:rec2}
    The same setting and test function as in \autoref{fig:rec}, but this time the object is rotated around $\be_2$ instead of $\be_1$.
		}
\end{figure}

\subsection{Perturbed rotation}
In practical applications, the rotation of the object is often a little perturbed.
For the following test, we consider
the rotation angle $\alpha(t)=t$ and the slightly moving rotation axis
$$
\bn(t)
= \left(\cos\left(\frac\pi8 \sin t\right), \sin\left(\frac\pi8 \sin t\right), 0 \right)^\tT
,$$
see \autoref{ex:mov-axis}.
This axis $\bn(t)$ is in the $r_1 r_2$ plane and it moves around $\be_1$.
The reconstruction results are depicted in \autoref{fig:rec3},
where we chose the resolution $N=160$ and thus $\rs\approx28.3$.
In these tests, the perturbation slightly improves the quality of the reconstructions, because the set of missing k-space data is different.
We note that we applied the discrete backpropagation \autoref{eq:bp} with Banach indicatrix $\operatorname{Card}(T_+^{-1})\equiv2$, even though it is 4 in a small region.
This still results in a good approximation.

\begin{figure}\centering
\begin{subfigure}[t]{.49\textwidth}
	\includegraphics[width=\textwidth,trim={44 30 36 23},clip]{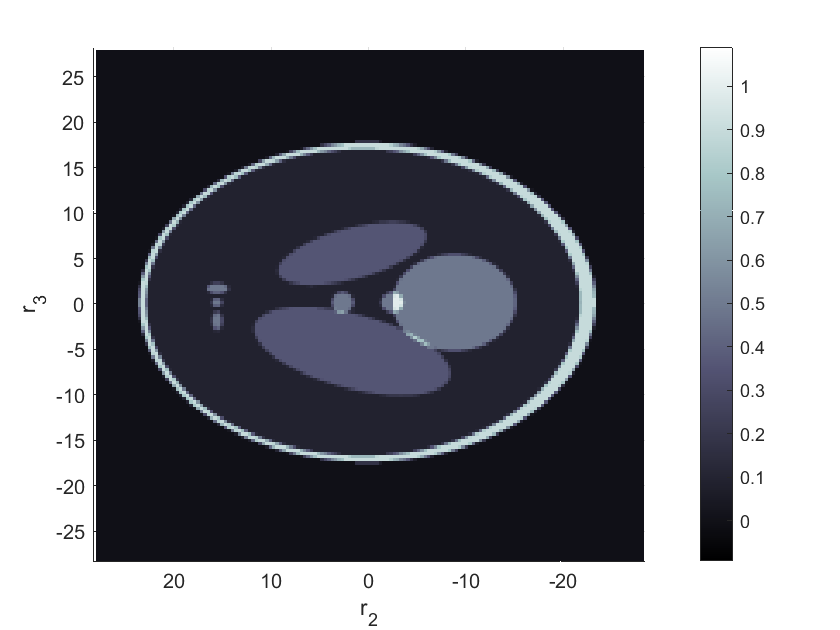}
	\caption{Phantom $\mathbf f_{160}^{\mathrm{av}}$ on a $160\times160\times160$ grid}
\end{subfigure}\\[\medskipamount]
\begin{subfigure}[t]{.49\textwidth}
	\includegraphics[width=\textwidth,trim={44 30 36 23},clip]{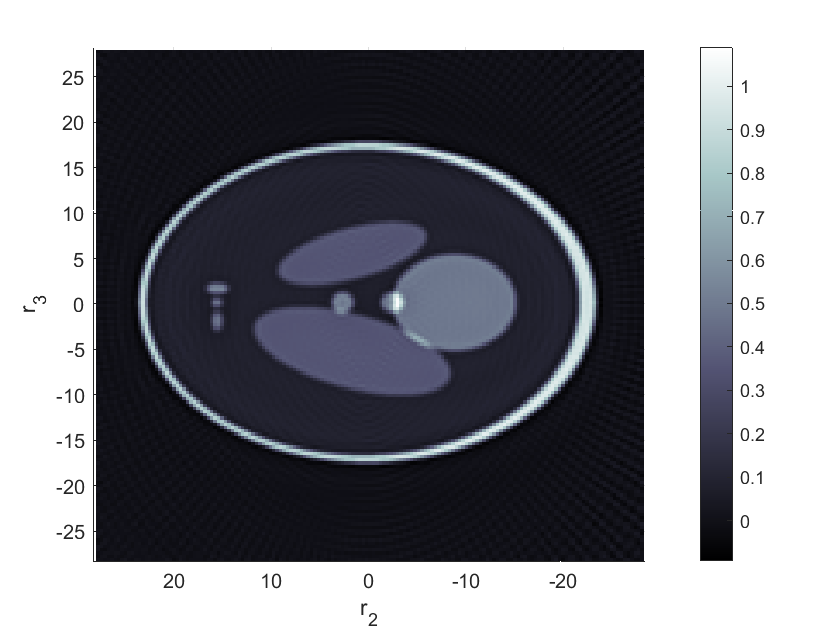}
	\caption{Inverse NDFT reconstruction with fixed rotation axis
		\\    PSNR 32.56, SSIM 0.892}
\end{subfigure}\hfill
\begin{subfigure}[t]{.49\textwidth}
\includegraphics[width=\textwidth,trim={44 30 36 23},clip]{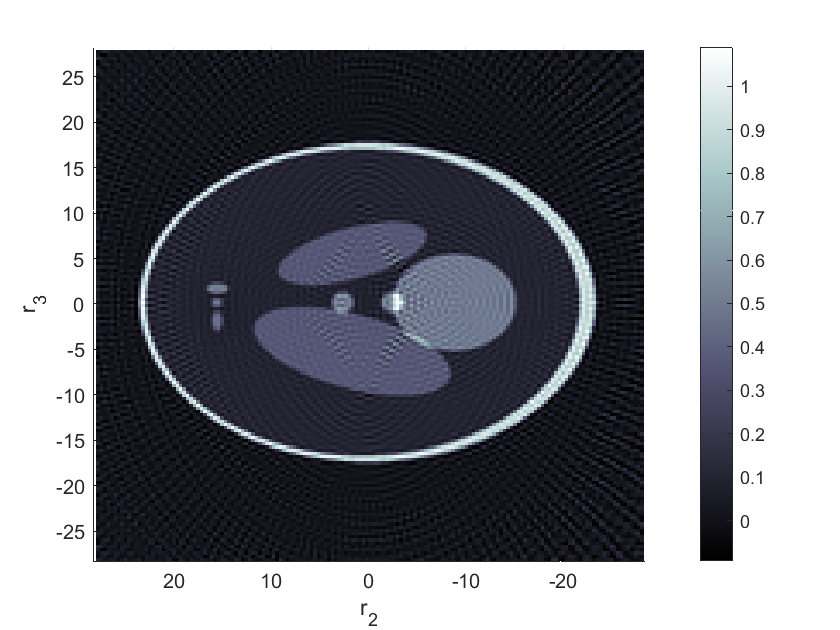}
\caption{Backpropagation with fixed rotation axis
	\\    PSNR 27.81, SSIM 0.422}
\end{subfigure}\\[\medskipamount]
\begin{subfigure}[t]{.49\textwidth}
\includegraphics[width=\textwidth,trim={44 30 36 23},clip]{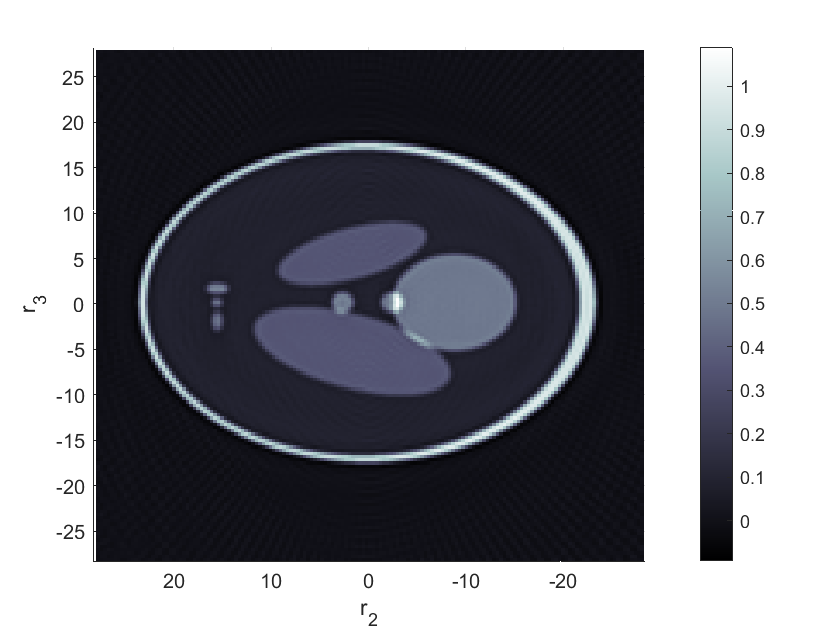}
\caption{Inverse NDFT reconstruction with slightly moving axis
	\\     PSNR 33.62, SSIM 0.934}
\end{subfigure}\hfill
\begin{subfigure}[t]{.49\textwidth}
\includegraphics[width=\textwidth,trim={44 30 36 23},clip]{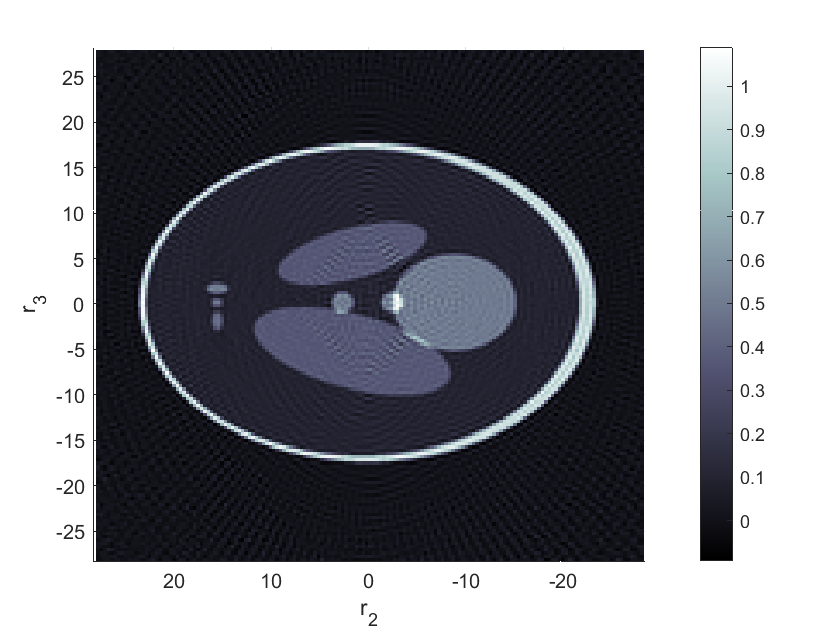}
\caption{Backpropagation with slightly moving axis
	\\     PSNR 30.92, SSIM 0.549}
\end{subfigure}
\caption{\label{fig:rec3}
	A Shepp--Logan phantom as test function (A) with the bigger resolution $N=160$.
	Reconstruction for the constant rotation axis $\bn(t)=\be_1$ with the inverse NDFT (B) and with the backpropagation (C).
	With the slightly moving rotation axis $\bn(t)$, the inverse NDFT (D) again produces a clearer image than the backpropagation~(E).
}
\end{figure}

In case of a half rotation around the fixed axis $\be_1$, we have an explicit formula of the Banach indicatrix, see \autoref{ex:half-rotation}.
The reconstruction shown in \autoref{fig:half-rotation} behaves slightly worse than before, due to the missing data.

\begin{figure}\centering
\begin{subfigure}[t]{.49\textwidth}
  \includegraphics[width=\textwidth,trim={44 30 36 23},clip]{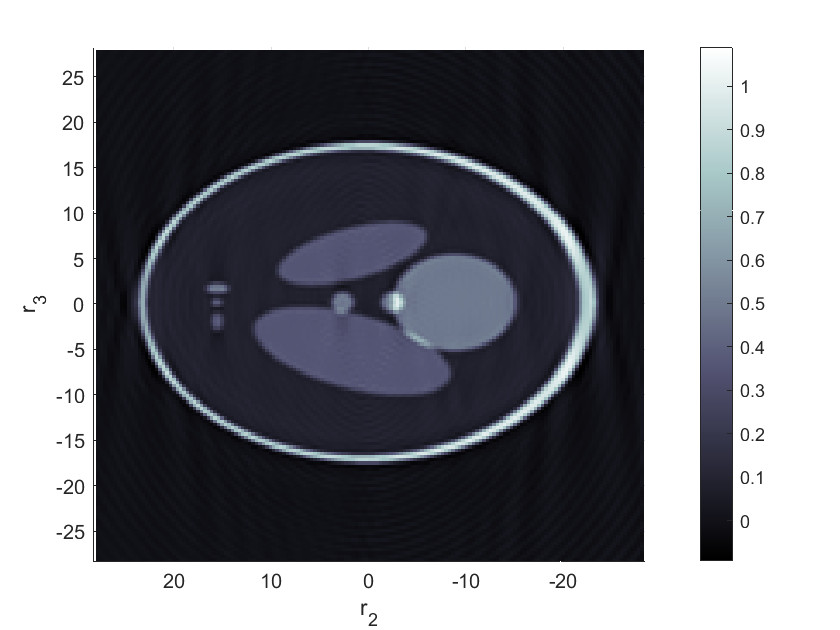}
  \caption{Inverse NDFT 
    \\     PSNR 30.80, SSIM 0.816}
\end{subfigure}\hfill
\begin{subfigure}[t]{.49\textwidth}
  \includegraphics[width=\textwidth,trim={44 30 36 23},clip]{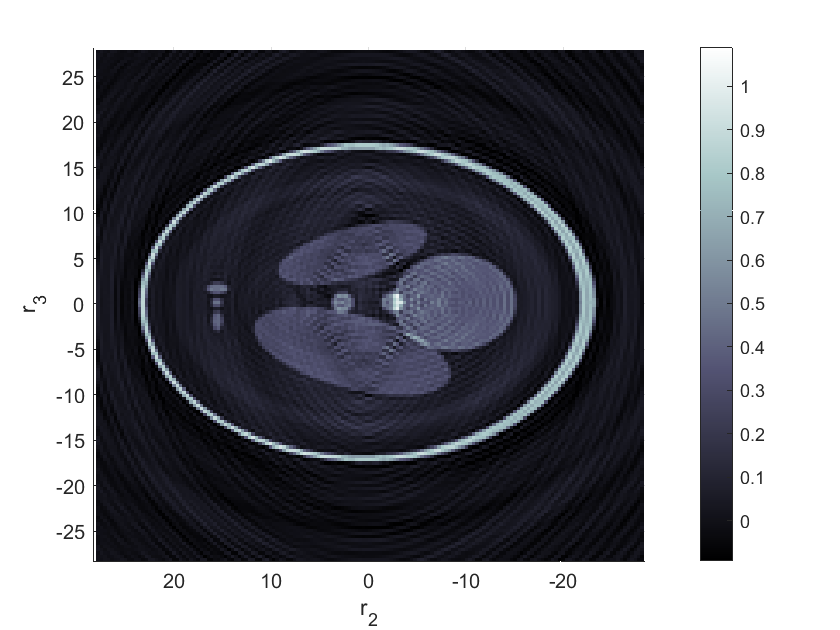}
  \caption{Backpropagation 
    \\     PSNR 26.01, SSIM 0.324}
\end{subfigure}
\caption{\label{fig:half-rotation}
  Reconstruction for a half rotation around $\be_1$, all other parameters are the same as in \autoref{fig:rec3}.
}
\end{figure}

\begin{remark}[Computation time]
The discrete backpropagation in \autoref{eq:bp} consists of one adjoint NDFT, whereas each iteration step of the CGNE method of the inverse NDFT requires about twice the computational effort: to perform both an NDFT and an adjoint NDFT.
Hence, the backpropagation algorithm is considerably faster than the inverse NDFT. 
In practice, this difference is a little smaller, since the NFFT software includes a node-dependent precomputation, which has to be done only once.
For the grid size of $N=160$, which corresponds to approximately 4 million grid points, the discrete backpropagation takes about 9.6 seconds, whereas 20 iteration steps of the CGNE algorithm take 46 seconds on an Intel Core i7-10700 CPU.
\end{remark}

\subsection{Noisy measurements}

In practical applications, the measurements are corrupted by noise.
For some noise level $\delta>0$, we consider the Gaussian white noise model
\begin{equation*}
  \mathbf g^\delta (\by)
  \coloneqq \ktran f(\by) + \delta\, \mathcal N(0,1)
  ,\quad \by\in\mathcal Y_N^+.
\end{equation*}
The CGNE method used for the inverse NDFT is a regularizer with the number of iterations as the regularization parameter.
Let us denote by $\mathbf f_k^\delta$ the $k$-th iterate of the CGNE method applied to $\mathbf g^\delta$.
The discrepancy principle \cite[Section 7.3]{EnHaNe96} states to choose $k$ such that the residual $\norm{\mathbf F_N \mathbf f_k^\delta - \mathbf g^\delta}$ is approximately $\delta$.
A popular parameter choice rule requiring no knowledge of the noise level $\delta$ is the L-curve method \cite{Han92}, where one chooses $k$ at the corner in the log-log plot of the residual $\norm{\mathbf F_N \mathbf f_k^\delta - \mathbf g^\delta}$ versus the norm of the approximate solution $\norm{\mathbf f_k^\delta }$.
As discrete norm, we take the root mean square 
$$ 
\norm{\mathbf g^\delta}
\coloneqq 
\left( \frac{1}{\abs{\smash{\mathcal Y_N^+}}} 
\sum_{\by\in\mathcal Y_N^+} \abs{\mathbf g^\delta(\by)}^2\right)^{1/2}.
$$
For an overview of parameter choice rules in the conjugate gradient method, we refer to \cite{HaPa07}.

In \autoref{tab:rec-noise}, we show the reconstruction error with different parameter choice rules.
We use the same test function as in \autoref{fig:rec3}, the resolution $N=160$, and the rotation around the $r_1$ axis. 
We compute the corner of the L-curve according to \cite{Han07}.
The L-curve method tends to overestimate the stopping index $k$,
whereas the discrepancy principle produces slightly better reconstructions.

\begin{table}
  \centering
  \begin{tabular}{c|c|c|c|c}
    Noise level $\frac{\delta}{\max\abs{\ktran f}}$
    & Best choice
    & L-curve
    & Discrepancy
    & Backpropagation
    \\\hline
    0.0 \%
    & 33.01 (100)
    & --
    & --
    & 27.81
    \\
    0.1 \%
    & 30.98 ~(18)
    & 27.80 (65)
    & 30.16 (13)
    & 23.25
    \\
    0.2 \%
    & 28.84 ~(13)
    & 25.89 (29)
    & 28.07 (10)
    & 18.49
    \\
    0.5 \%
    & 25.20 ~~(8)
    & 18.81 (25)
    & 24.07 ~(6)
    & 11.04
    \\
    1.0 \%
    & 22.58 ~~(6)
    & 20.78 (3)
    & 20.78 ~(3)
    & ~5.03
\end{tabular}
\caption{ \label{tab:rec-noise}
  PSNR for with different noise levels $\delta$ with the L-curve method and discrepancy principle for choosing the number $k$ of CGNE iteration, which is displayed in parantheses.
  The best choice serves as a benchmark; it is attained at the iteration index $k$ which maximizes 
	$\operatorname{PSNR}(\mathbf{f}^{\mathrm{av}},\mathbf{f}_k^\delta)$ among $k=1,\dots,100$, where we know the ground truth $\mathbf{f}^{\mathrm{av}}$.
  The backpropagation contains no regularization and thus works only well with very small noise.
  The function $f$ is the same as in \autoref{fig:rec3} with constant rotation axis~$\be_1$ and resolution $N=160$.
  }
\end{table}

\section{Appendix} \label{appendix}
\subsection{Distributions and the Fourier Transform}\label{ss:ktran}
This section collects several results about distributions and the Fourier transform. We refer to \cite{Gru09,Hoe03,PlPoStTa18,ReeSim75} for more details.

\begin{definition}[Test function]
The space $\mathcal{D}(\R^n)$ of test functions is the set of all infinitely differentiable functions $\phi:\R^n\to\C$ with compact support equipped with the inductive limit topology.
\end{definition} 

\begin{definition}[Distribution]
The space $\mathcal{D}'(\R^n)$ of distributions is the topological dual of $\mathcal{D}(\R^n)$, i.e. the space of all functionals $v:\mathcal{D}(\R^n)\to\C$ such that
\begin{gather*}
v[a\phi + b\psi] = av[\phi] + bv[\psi]
\end{gather*}
for all $a,b\in \C$ and $\phi,\psi\in\mathcal{D}(\R^n)$ and
\begin{equation*}
v[\phi_n]\to v[\phi]
\end{equation*}
whenever $\phi_n\to\phi$ in $\mathcal{D}(\R^n)$.
\end{definition}

\begin{example}
Every locally absolutely integrable function $v \in L^1_{\mathrm{loc}} (\R^n)$ can be identified with a distribution via
\begin{equation*}
\phi \mapsto \int_{\R^n} v(\br) \phi(\br) \, d\br.
\end{equation*}
It is common to use the same symbol for the function and the associated distribution.
\end{example}

By $L^p(\Omega)$, $p \in[1,\infty]$, 
we  denote the Banach space of (equivalence classes of) complex-valued functions
with finite norm
\begin{align*}
\|f\|_{L^p} &\coloneqq \left( \int_\Omega |f(\br)|^p \, d \br \right)^\frac1p, \quad p \in [1,\infty),\\
\|f\|_{L^\infty} &\coloneqq \mathrm{esssup}_{\br \in \Omega} |f(\br)|.
\end{align*}                                       
If $\Omega \subset \mathbb \R^n$ is bounded, then these spaces are continuously embedded, i.e., for $1\le p \le q$,
we have
\begin{equation}\label{embedding}
L^q(\Omega) \hookrightarrow L^p(\Omega).
\end{equation}
The functions of $\mathcal{D}(\R^n)$ form a dense set in $L^p(\R^n)$, $p \in[1,\infty)$.

\begin{definition}[Fourier transform on $L^1(\R^n)$] 
The Fourier transform of a function $g \in L^1(\R^n)$ is defined by
	\begin{equation} \label{eq:ft}
		\mathcal{F}g(\bk) \coloneqq (2\pi)^{-\frac{n}{2}} \int_{\R^n} g(\br) \e^{-\i \bk \cdot \br} d\br
	\end{equation}
for all $\bk \in \R^n$.
\end{definition}

The Fourier transform $\mathcal{F}: L^1(\R^n) \rightarrow C_0(\R^n)$ is a linear, continuous operator into the space
$C_0(\R^n)$ of continuous functions vanishing at infinity with operator norm 
$\|\mathcal{F}\|_{L^1 \rightarrow C_0} = (2\pi)^{-\frac{n}{2}}$.
In contrast, the Fourier transform on the two spaces introduced next, maps again onto the same space.

\begin{definition}[Schwartz space]\label{de:schwartz}
The Schwartz space $S(\R^n)$ consists of all infinitely differentiable functions $\phi:\R^n \to \C$ such that
\begin{equation*}
	p_{\alpha,\beta}(\phi) \coloneqq \sup_{\br \in \R^n} |\br^\alpha \partial^\beta \phi(\br)| < +\infty
\end{equation*}
for all multi-indices $\alpha, \beta \in \N_0^n$. The topology on $S(\R^n)$ is induced by the seminorms $p_{\alpha,\beta}$.
\end{definition}

\begin{definition}[Tempered distribution]
The space $S'(\R^n)$ of tempered distributions is the topological dual of $S(\R^n)$.
\end{definition}

\begin{example} \label{ex:L1loc}
Every function $v \in L^1_{\mathrm{loc}} (\R^n)$ that is polynomially bounded for $\|{\br} \| \to \infty$ can be identified with a tempered distribution via
\begin{equation*}
	\phi \mapsto \int_{\R^n} v(\br) \phi(\br) \, d\br.
\end{equation*}
\end{example}

For the following result note that $S(\R^n) \subset L^1(\R^n)$.

\begin{proposition}\label{thm:homeo_S}
The Fourier transform is a linear, bijective, continuous operator
$\mathcal{F}: S(\R^n) \rightarrow S(\R^n)$. It has a
continuous inverse defined by
\begin{gather*}
		\mathcal{F}^{-1}\phi(\br) \coloneqq (2\pi)^{-\frac{n}{2}} \int_{\R^n} \phi(\bk) \e^{\i \bk \cdot \br} d\bk
\end{gather*}
for all $\phi \in S(\R^n)$ and $\br \in \R^n$.
\end{proposition}

Observe that
\begin{equation}\label{eq:phiFpsi}
	\int_{\R^n} \phi(\br)\, \mathcal{F} \psi(\br) \, d\br = \int_{\R^n} \mathcal{F} \phi(\br)\, \psi(\br) \, d\br,
\end{equation}
holds for all $\phi, \psi \in S(\R^n)$. Combined with the fact that, for every $v\in S'(\R^n)$, the map $\phi \mapsto v[\ktran \phi]$ is another tempered distribution we are led to the following extension of the Fourier transform.

\begin{definition}[Fourier transform on $S'(\R^n)$]\label{de:F} The Fourier transform of $v \in S'(\R^n)$ is defined by
	\begin{equation*}
		\mathcal{F}v[\phi] \coloneqq v[\mathcal{F}\phi]
	\end{equation*}
for all $\phi \in S(\R^n)$.
\end{definition}

With this definition, \autoref{thm:homeo_S} carries over to $S'(\R^n)$.

\begin{proposition}
The Fourier transform is a linear, bijective, continuous operator
$\mathcal{F}: S'(\R^n) \rightarrow S'(\R^n)$.
It  has a continuous inverse defined by
\begin{gather*}
	\mathcal{F}^{-1} v[\phi] = v[\mathcal{F}^{-1}\phi]
\end{gather*}
for all $v \in S'(\R^n)$ and $\phi \in S(\R^n)$.
\end{proposition}

\begin{remark}\label{re:extension}
Many other operations can be extended from $S(\R^n)$ to $S'(\R^n)$ in a way similar to  \autoref{de:F}. 
Suppose $A:S(\R^n) \to S(\R^n)$ is linear and bounded, 
and that there is another linear, bounded operator 
$B:S(\R^n) \to S(\R^n)$ such that $\int_{\R^n} (A\phi) \psi = \int_{\R^n} \phi (B\psi)$ 
for all Schwartz functions $\phi$ and $\psi$. Then $A$ can be uniquely extended to $S'(\R^n)$ 
by setting $Au[\phi] = u[B\phi]$. See \cite[Rem.\ 5.15]{Gru09} for more details.
\end{remark}
Two further operations which can be extended from $S(\R^n)$ to $S'(\R^n)$ in the way explained above are multiplication and convolution with a Schwartz function.
Both $\phi \mapsto \psi \phi$ and $\phi \mapsto \psi * \phi$ map $S(\R^n)$ continuously into itself.
Moreover, regarding convolution note that we have
\begin{equation}\label{eq:M}
	\int_{\R^n} (\psi * \eta)(\br) \phi(\br)\, d\br = \int_{\R^n} \eta(\br) (M\psi * \phi)(\br)\, d\br
\end{equation}
for all $\phi,\psi,\eta \in S(\R^n)$, where the operator $M: S(\R^n) \rightarrow S(\R^n)$ 
is given by $M\phi(x) = \phi(-x).$ This gives rise to the following definition.
\begin{definition} \textbf{(Convolution of a Schwartz function with a tempered distribution)} 
The convolution of $\psi \in S(\R^n)$ with $v \in S'(\R^n)$ is defined by
\begin{equation*}
	(\psi * v) [\phi] \coloneqq v[M\psi * \phi]
\end{equation*}
for all $\phi \in S(\R^n)$.
\end{definition}

The following result relates the operations of multiplication and convolution by means of the Fourier transform.

\begin{theorem}[Convolution Theorem]\label{thm:convolution_theorem}
	For all $\phi \in S(\R^n)$ and all $v \in S'(\R^n)$, we have
	\begin{equation*}
		\mathcal{F}(\phi * v) = (2\pi)^{\frac{n}{2}} \mathcal{F}\phi \mathcal{F}v \quad \text{and} \quad 
		\mathcal{F}\phi * \mathcal{F}v = (2\pi)^{\frac{n}{2}}\mathcal{F}(\phi v) .
	\end{equation*}
\end{theorem}

Note that \autoref{thm:convolution_theorem} remains true, if we replace $\mathcal{F}$ by $\mathcal{F}^{-1}$. 

\subsection{Partial Fourier transforms} \label{ss:Fpartial}
In the following, we introduce partial Fourier transforms, 
and show that they are well-defined on $S'(\R^n)$.
In other words, we prove that an $m$-dimensional Fourier transform  $(m < n)$ 
is well-defined for $n$-dimensional tempered distributions and that it possesses an appropriate convolution property.

\begin{definition}[Partial Fourier transforms on $S(\R^n)$]\label{de:F_I}
For $j \in \{1,\ldots,n\}$, we define the partial Fourier transform $\mathcal{F}_{j}$ of  $\phi\in S(\R^n)$
by
\begin{equation}\label{eq:F_I}
\mathcal{F}_{j} \phi(r_1,\ldots, r_{j-1}, k_j, r_{j+1},\ldots r_n) 
=  (2\pi)^{-\frac{1}{2}} \int_{\R} \phi(r_1,\ldots, r_{j-1}, s, r_{j-1},\ldots r_n) \e^{-\i k_j  s} d s.
\end{equation}
More generally, for an index set $I = \{j_1,\ldots, j_m\}\subset \{1,\ldots,n\}$, the partial Fourier transform  $\mathcal{F}_I\phi$ of $\phi\in S(\R^n)$
is defined by
\begin{equation}\label{eq:F_I_1}
	\mathcal{F}_I \phi \coloneqq \mathcal{F}_{j_m} \cdots \mathcal{F}_{j_1} \phi. 
\end{equation}
\end{definition}
Note that, for fixed $r_1,\ldots, r_{j-1}, r_{j+1},\ldots r_n$, the map $r_j \mapsto \phi(\br)$, defines a Schwartz function on $\R$. This function is then also in $L^1(\R)$ and therefore the integral in \autoref{eq:F_I} is well-defined. Moreover, by the Fubini-Tonelli theorem the order in which the $\ktran_{j_i}$ appear in \autoref{eq:F_I_1} does not matter. Finally, we point out that $\ktran_{1,\dots,n} = \ktran$.

\begin{proposition} 
Partial Fourier transforms are linear, bijective, continuous operators $\mathcal{F}_I: S(\R^n) \to S(\R^n)$. They have continuous inverses defined by
\begin{equation*}
	\mathcal{F}^{-1}_I \phi \coloneqq \mathcal{F}^{-1}_{j_m} \cdots \mathcal{F}^{-1}_{j_1} \phi,
\end{equation*} 
where
\begin{equation*}
	\mathcal{F}^{-1}_{j} \phi(r_1,\ldots, r_{j-1}, k_j, r_{j+1},\ldots r_n) 
=  (2\pi)^{-\frac{1}{2}} \int_{\R} \phi(r_1,\ldots, r_{j-1}, s, r_{j-1},\ldots r_n) \e^{\i k_j  s} d s.
\end{equation*}
\end{proposition}
\begin{proof}
This result can be shown in essentially the same way as \autoref{thm:homeo_S}. The main steps are as follows.

First, from \autoref{eq:F_I} and \autoref{eq:F_I_1} we deduce that the function $\mathcal{F}_I\phi$ is bounded. Second, it is infinitely differentiable and the operator $\mathcal{F}_I$ exchanges differentiation and multiplication with polynomials in the expected way. More specifically, and assuming that $I=\{1,\ldots,m\}$ to keep the notation simple, we have
\begin{equation}\label{eq:ldftproperty}
	(\bk,\br)^\alpha \partial^\beta \mathcal{F}_I \phi =
	(-\i)^{\abs{\alpha_1}+\abs{\beta_1}} \mathcal{F}_I \left( \partial_{\bs}^{\alpha_1} \br^{\alpha_2} \bs^{\beta_1} \partial_{\br}^{\beta_2} \phi \right)
\end{equation}
for all $\bk \in \R^m$, $\br \in \R^{n-m}$ and all multi-indices $\alpha = (\alpha_1,\alpha_2), \beta = (\beta_1,\beta_2) \in \N_0^m \times \N_0^{n-m}$. These two facts imply that $\mathcal{F}_I\phi \in S(\R^n)$.

Next, exploiting \autoref{eq:ldftproperty}, we can show that for each pair of multi-indices 
$\alpha, \beta$ there exists another pair $\alpha',\beta'$ and a $C>0$ such that 
$p_{\alpha,\beta}(\mathcal{F}_I\phi) \le C p_{\alpha',\beta'}(\phi)$. 
This shows that $\mathcal{F}_I$ is continuous on $S(\R^n)$.

Finally, note that $\bk \mapsto \mathcal{F}_I\phi(\bk,\br)$ is nothing but the regular 
Fourier transform of the Schwartz function 
$\bs \mapsto \phi(\bs,\br)$. Therefore $\mathcal{F}_I$ has an inverse. 
It is continuous on $S(\R^n)$ for the same reasons $\mathcal{F}_I$ is.
\end{proof}

Since \autoref{eq:phiFpsi} remains valid if we replace $\mathcal{F}$ by $\mathcal{F}_I$, 
we can extend the partial Fourier transforms to $S'(\R^n)$.
\begin{definition}[Partial Fourier transforms on $S'(\R^n)$]\label{de:F_I_S'}
For every $v \in S'(\R^n)$, we define $\mathcal{F}_I v$ by
	\begin{equation*}
		\mathcal{F}_I v[\phi] \coloneqq v[\mathcal{F}_I\phi]
	\end{equation*}
for all $\phi \in S(\R^n)$.
\end{definition}

\begin{proposition}\label{thm:F_I_homeo}
Partial Fourier transforms are linear, bijective, continuous operators $\mathcal{F}_I: S'(\R^n) \to S'(\R^n)$. They have continuous inverses defined by
\begin{gather*}
	\mathcal{F}_I^{-1} v[\phi] = v[\mathcal{F}_I^{-1}\phi]
\end{gather*}
for all $v \in S'(\R^n)$ and all $\phi \in S(\R^n)$.
\end{proposition}

\begin{proof}
The continuity of the extension of $\ktran_I$ to $S'(\R^n)$ 
follows from \autoref{re:extension}. 
Analogously, $\ktran_I^{-1}$ can be extended to a continuous operator on $S'(\R^n)$. 
These two extensions are still inverse to each other, because 
$\ktran_I \ktran_I^{-1} v [\phi ] =   v [\ktran_I^{-1} \ktran_I \phi ] = v [ \phi ] = \ktran_I^{-1} \ktran_I v [\phi ] $.
\end{proof}

\begin{remark} It follows from \autoref{de:F_I} that partial Fourier transforms can be composed in the following way. 
Suppose $I$ and $J$ are disjoint subsets of $\{1,\ldots,n\}$. 
Then an application of the Fubini-Tonelli theorem shows 
that $\mathcal{F}_I \mathcal{F}_J \phi = \mathcal{F}_{I\cup J} \phi$ 
for every $\phi \in S(\R^n)$. 
This property immediately carries over to the distributional setting
\begin{equation*}
	\mathcal{F}_I \mathcal{F}_J v[\phi] = v[\mathcal{F}_J \mathcal{F}_I\phi] = v[\mathcal{F}_{I\cup J} \phi] = \mathcal{F}_{I\cup J}  v[\phi]
\end{equation*}
for all $v \in S'(\R^n)$ and all $\phi \in S(\R^n)$.
In the proof of \autoref{thm:fourierdiffraction}, 
where $n=3$, 
we make use of the special case $ \mathcal{F}_3 \mathcal{F}_{1,2} = \mathcal{F}$ 
implying that
\begin{equation}\label{eq:F12=F3-1F}
	\mathcal{F}_{1,2} =  \mathcal{F}^{-1}_3 \mathcal{F}.
\end{equation}
\end{remark}

\begin{definition}[Partial convolution on $S(\R^n)$]
Let $j\in\{1,\ldots,n\}$ and $\phi, \psi \in S(\R^n)$. 
We denote by $\phi \stackrel{j}{*} \psi$ the convolution of $\phi$ and $\psi$ along the $j$-th coordinate, i.e., 
\begin{equation*}
	(\phi \stackrel{j}{*} \psi) (\br) \coloneqq \int_\R \phi (r_1,\ldots, r_{j-1},s,r_{j+1},\ldots,r_n) \psi(r_1,\ldots, r_{j-1},r_j-s,r_{j+1},\ldots,r_n)  \, ds.
\end{equation*}
\end{definition}

As is the case for ordinary convolution, partial convolution with a Schwartz function is a continuous operation on $S(\R^n)$. 

\begin{proposition}
For every $\psi \in S(\R^n)$ and $j\in\{1,\ldots,n\}$ the operation $\phi \mapsto \psi \stackrel{j}{*} \phi$ maps $S(\R^n)$ continuously into itself.
\end{proposition}

\begin{proof}
We sketch the main steps of this proof. First, since we can differentiate under the integral sign, the function $\psi \stackrel{j}{*} \phi$ is infinitely differentiable for all $\phi,\psi \in S(\R^n)$. Furthermore, by means of a generalized product rule one can show that for every multi-index $\alpha$, the derivative $\partial^\alpha ( \psi \stackrel{j}{*} \phi )$ is a linear combination of partial convolutions of derivatives of $\phi$ and $\psi$. Combining this with the fact that $\br^\beta (\chi \stackrel{j}{*} \eta)$ is a bounded function for all multi-indices $\beta$ and Schwartz functions $\chi, \eta$, we find that $\br^\beta \partial^\alpha ( \psi \stackrel{j}{*} \phi )$ is bounded too. In fact, it has an upper bound of the form $Cp_{\alpha',\beta'}(\phi)$, where $C\ge 0$ depends on $\psi.$ This finishes the proof.
\end{proof}

Partial convolution is not only continuous, but also satisfies an identity analogous to \autoref{eq:M}, namely
\begin{equation*}
	\int_{\R^n} (\eta \stackrel{j}{*} \phi)(\br) \psi(\br)  \, d\br = \int_{\R^n} \phi (\br) (M_j\eta \stackrel{j}{*} \psi)(\br) \, d\br
\end{equation*}
for all $\phi,\psi,\eta \in S(\R^n)$, where $M_j\phi(\br) = \phi(r_1,\ldots,r_{j-1},-r_j, r_{j+1},\ldots, r_n).$ Thus we can extend it to a continuous operation on $S'(\R^3)$ in the following way, recall \autoref{re:extension}.

\begin{definition} \textbf{(Partial convolution of a Schwartz function with a tempered distribution)}\enspace
The partial convolution of $\psi \in S(\R^n)$ with $v \in S'(\R^n)$ is defined by
\begin{equation*}
	(\psi \stackrel{j}{*} v) [\phi] \coloneqq v [M_j\psi \stackrel{j}{*} \phi]
\end{equation*}
for all $\phi \in S(\R^n)$.\label{de:ldc}
\end{definition}

Finally, we have all prerequisites 
for formulating a one-dimensional convolution theorem for tempered distributions on $\R^n$.

\begin{theorem}[Partial Convolution Theorem on $S'(\R^n)$]\label{thm:ldctS'}
For all $\phi \in S(\R^n)$ and $v \in S'(\R^n)$, we have
\begin{equation} \label{eq:ldctS'}
	\mathcal{F}_j (\phi \stackrel{j}{*} v) = (2\pi)^{\frac{1}{2}} \mathcal{F}_j \phi \mathcal{F}_j v,
	\quad \text{and} \quad 
	\mathcal{F}_j \phi \stackrel{j}{*} \mathcal{F}_j v = (2\pi)^{\frac{1}{2}} \mathcal{F}_j (\phi v) 
\end{equation}
The same identities hold for $\mathcal{F}_j^{-1}$ instead of $\mathcal{F}_j$.
\end{theorem}
\begin{proof}
First we observe that a corresponding one-dimensional convolution theorem for Schwartz functions, that is,
\begin{equation} \label{eq:ldctS}
	\mathcal{F}_j (\phi \stackrel{j}{*} \psi ) 
	= (2\pi)^{\frac{1}{2}} \mathcal{F}_j \phi\mathcal{F}_j \psi \quad \text{ and } \quad
	\mathcal{F}_j \phi \stackrel{j}{*} \mathcal{F}_j \psi = (2\pi)^{\frac{1}{2}} \mathcal{F}_j (\phi\psi)
\end{equation}
for all $\phi,\psi \in S(\R^n)$ holds true. Indeed, letting the coordinates $r_i$ for $i \neq j$ be fixed, 
\autoref{eq:ldctS} is nothing but a standard convolution theorem for the one-dimensional 
Schwartz functions $r_j \mapsto \phi(\br)$ and $r_j \mapsto \psi(\br)$.

Next, it follows directly from \autoref{de:F_I} and \autoref{de:ldc} that
\begin{equation*}
	\mathcal{F}_j (\phi \stackrel{j}{*} v) [\psi]
	=	\phi \stackrel{j}{*} v [\mathcal{F}_j \psi]
	=	v [M_j\phi \stackrel{j}{*} \mathcal{F}_j \psi].
\end{equation*}
The second identity in \autoref{eq:ldctS} implies that
\begin{equation*}
	M_j\phi \stackrel{j}{*} \mathcal{F}_j \psi = (2\pi)^{\frac{1}{2}}  \mathcal{F}_j \left( (\mathcal{F}_j^{-1} M_j \phi) \psi \right)
	= 
	(2\pi)^{\frac{1}{2}}  \mathcal{F}_j \left( (\mathcal{F}_j \phi) \psi \right), 
\end{equation*}
where we have also exploited the fact that $\mathcal{F}_j = \mathcal{F}_j^{-1}M_j$. Combining the previous two equalities gives
\begin{equation*}
	\mathcal{F}_j (\phi \stackrel{j}{*} v) [\psi]
	= (2\pi)^{\frac{1}{2}} v[\mathcal{F}_j \left( (\mathcal{F}_j \phi) \psi \right)]
	= (2\pi)^{\frac{1}{2}} \mathcal{F}_j v[  (\mathcal{F}_j \phi) \psi ]
	= (2\pi)^{\frac{1}{2}} (\mathcal{F}_j \phi) \mathcal{F}_j v[ \psi ].
\end{equation*} 
In the last equality we have used the fact that multiplication of a tempered distribution with a Schwartz function 
is defined by 
$\phi v[\psi] = v[\phi \psi]$.  
This finishes the first part of \autoref{eq:ldctS'}. 
The second part follows analogously, as do the corresponding formulas for the inverse $\mathcal{F}^{-1}_j$.
	$\Box$

\subsection{Proof of \autoref{thm:fourierdiffraction}} \label{sec:fdt}
	In order to prove \autoref{thm:fourierdiffraction}, we have to characterize the partial Fourier transform $\ktran_{1,2}G$, where the subscripts indicate that the transform is taken with respect to $r_1$ and $r_2$ and
	\begin{equation}\label{eq:green}
	G(\br) \coloneqq \frac{\e^{\i k_0\norm{\br}}}{4\pi\norm{\br}}
	\end{equation}
	is the Green's function of the Helmholtz operator $-(\Delta+k_0^2)$ which satisfies the Sommerfeld radiation condition, see \cite[Chap.\ 2]{ColKre13}. Since $G(\cdot,\cdot,r_3)$ is not in $L^1(\R^2)$ for any $r_3\in \R$,  we cannot compute this transform $\ktran_{1,2} G$ as an ordinary Fourier integral. Instead, in \autoref{thm:partfourierG} we calculate $\ktran_{1,2}G$ in the distributional sense. Mathematical details about distributions, (partial) Fourier transforms and convolutions can be found in \autoref{ss:ktran} and \autoref{ss:Fpartial}.

	We define
	\begin{equation*}
	G_\epsilon(\br) \coloneqq \e^{-\epsilon \norm{\br}} G(\br), \quad \epsilon >0,
	\end{equation*}
	and note that $G_\epsilon \to G$ in $S'(\R^3)$ for $\epsilon \to 0$. We will also use the abbreviation 
	\begin{equation} \label{eq:kappa_eps}
	\kappa_\epsilon \coloneqq \sqrt{ (k_0+ \i \epsilon)^2-k_1^2 - k_2^2}
	\end{equation}
	to denote the principal square root of $(k_0+ \i \epsilon)^2-k_1^2 - k_2^2$, 
	that is, the root with positive imaginary part.
	
	\begin{lemma} \label{thm:partfourierG}
		The partial Fourier transform $\ktran_{1,2}G \in S'(\R^3)$ is given by
		\begin{equation*}
		\ktran_{1,2}G [\phi]
		= \lim_{\epsilon \to 0} \ktran_{1,2}G_\epsilon [\phi]
		= \lim_{\epsilon \to 0} \int_{\R^3} \frac{\i \e^{\i  \kappa_\epsilon \abs{r_3}}}{4\pi \kappa_\epsilon} \phi(k_1,k_2,r_3)\, d(k_1,k_2,r_3)
		\end{equation*}
		for all Schwartz functions $\phi \in S(\R^3)$.
	\end{lemma}
	\begin{proof}
		
		Since $\ktran_{1,2}$ is continuous on $S'(\R^3)$, recall \autoref{thm:F_I_homeo}, we also have $\ktran_{1,2}G_\epsilon \to \ktran_{1,2}G$ in $S'(\R^3)$. It remains to calculate $\ktran_{1,2}G_\epsilon$. We do so in two steps, according to \autoref{eq:F12=F3-1F}.
		
		First, exploiting the fact that $G_\epsilon$ is a radial function, see \cite[Rem.\ 4.32]{PlPoStTa18}, we obtain for the 3D Fourier transform
		\begin{equation*} 
		\ktran G_\epsilon(\bk)
		= \norm{\bk}^{-1} (2\pi)^{-\frac32} \int_0^\infty \e^{s(-\epsilon + \i k_0)}	\sin(s\norm{\bk})\, ds,
		\end{equation*} 
		for all $\bk \in \R^3$.
		The integral on the right-hand side can be calculated via integration by parts as
		\begin{equation*}
		\ktran G_\epsilon(\bk)
		= \frac{(2\pi)^{-\frac32}}{\norm{\bk}^2-( k_0+ \i \epsilon)^2 }
		= \frac{(2\pi)^{-\frac32}}{k_3^2 - \kappa_\epsilon^2}.
		\end{equation*}
		Second, for $a \in \mathbb C$ with $\operatorname{Re} a>0$, formula 17.23.14 in \cite{GrRy07} states that
		$$
		\int_\R \frac{\e^{\i k_3r_3}}{k_3^2 + a^2}dk_3
		= \pi \frac{\e^{-a \abs{r_3}}}{a}.
		$$ 
		With respect to the above integral, 
		we use $a:=-\i \kappa_\epsilon$, which fulfills $\operatorname{Re} a>0$ since $\operatorname{Im} \kappa_\epsilon>0$.
		Then we obtain 
		\begin{equation} \label{star-star}
		\ktran_{1,2} G_\epsilon = \ktran_3^{-1} \ktran G_\epsilon = 
		(2\pi)^{-2} \int_{\R}\frac{\e^{\i k_3r_3}}{k_3^2 - \kappa_\epsilon^2}dk_3
		= -\frac{1}{4\pi} \frac{\e^{\i\kappa_\epsilon\abs{r_3}}}{\i\kappa_\epsilon}.
		\end{equation}
	\end{proof}

\emph{Proof of \autoref{thm:fourierdiffraction}:}
	Let $g\in L^p(\R^3)$, $p>1$, with $\supp(g)\subset\mathcal{B}_{r}$.
	Then, by the embedding in \autoref{embedding} and  density of $\mathcal {D}(\R^3)$ in $L^p(\R^3)$, $p \in [1,\infty)$,  we can find a sequence of functions $g_n \in\mathcal{D}(\R^3)$ with $\supp g_n \in \mathcal{B}_{r}$,
	such that $g_n\to g$ in $L^q(\R^3)$, for $q\in[1,p]$ if $p\in (1,\infty)$ and for $q\in[1,\infty)$ otherwise, as $n\to\infty$. For each $g_n$, consider
	\begin{equation*}
	(\Delta+k_0^2)u_n=-g_n
	\end{equation*}
	with the Sommerfeld radiation condition. The unique solution $u_n$ is given by the convolution $u_n = g_n*G$, see \cite[Chap.\ 2]{ColKre13}.
	From \autoref{eq:F12=F3-1F} and the convolution theorems \ref{thm:convolution_theorem} and \ref{thm:ldctS'}, for the 3D Fourier transform $\ktran $ and the partial Fourier transform $\ktran_3$, 
	it follows that
	\begin{equation} \label{eq:fourier}
	\begin{aligned}
	\ktran_{1,2} u_n&=\ktran_3^{-1} \ktran (g_n * G)
	= (2\pi)^{\frac32}\ktran_3^{-1} (\ktran g_n \ktran G)
	= 2\pi   \left(\ktran_3^{-1}\ktran g_n \right) \stackrel{3}{*} \left(\ktran_3^{-1}\ktran G \right)\\
	& = 2\pi   (\ktran_{1,2}g_n) \stackrel{3}{*}(\ktran_{1,2}G),
	\end{aligned}
	\end{equation}
	where $\stackrel{3}{*}$ denotes the partial convolution with respect to the third coordinate. 
	Now, for every $\phi \in S(\R^3)$, it follows by continuity of partial convolutions on $S^\prime(\R^3)$ 
	that
	\begin{equation*} 
	\begin{aligned}
	\ktran_{1,2}u_n[\phi]
	&= 2\pi\left(\ktran_{1,2}g_n \stackrel{3}{*}\ktran_{1,2}G\right)[\phi] = 2\pi \lim_{\epsilon \to 0} \left(\ktran_{1,2}g_n \stackrel{3}{*}\ktran_{1,2}G_\epsilon\right)[\phi]\\
	&= 2\pi \lim_{\epsilon \to 0} \int_{\R^3} \ktran_{1,2}G_\epsilon[M_3\ktran_{1,2}g_n \stackrel{3}{*} \phi] \, d(k_1,k_2,r_3),
	\end{aligned}
	\end{equation*}
	and by Fubini's theorem and \autoref{star-star} further
	\begin{equation*} 
	\begin{aligned}
	\ktran_{1,2}u_n[\phi]
	&= \frac{\i}{2}\lim_{\epsilon \to 0} \int_{\R^3}  \frac{\phi}{\kappa_\epsilon}\int_{\R}\e^{\i\kappa_\epsilon\left|r_3-x\right|}\ktran_{1,2}g_n(k_1,k_2,x)\,dx\,d(k_1,k_2,r_3).
	\end{aligned}
	\end{equation*}
	In order to change integration and limit, we will apply Lebesgue's dominated convergence theorem twice. 
	First, noting that $\ktran_{1,2}g_n(k_1,k_2,\cdot)\in L^{1}(\R)$ we conclude by Lebesgue's theorem that
	\begin{equation*} 
	\lim_{\epsilon\to 0}\int_{\R}\e^{\i\kappa_\epsilon\left|r_3-x\right|}\ktran_{1,2}g_n(k_1,k_2,x)\,dx=\int_{\R}\e^{\i\kappa\left|r_3-x\right|}\ktran_{1,2}g_n(k_1,k_2,x)\,dx,
	\end{equation*}
	for all $(k_1,k_2,r_3)\in\R^3$. It is then immediate that
	\begin{equation}\label{eq:point1}
	\lim_{\epsilon\to 0}\frac{\phi(k_1,k_2,r_3)}{\kappa_\epsilon(k_1,k_2)}\int_{\R}\e^{\i\kappa_\epsilon\left|r_3-x\right|}\ktran_{1,2}g_n(k_1,k_2,x)\,dx=\frac{\phi(k_1,k_2,r_3)}{\kappa(k_1,k_2)}\int_{\R}\e^{\i\kappa\left|r_3-x\right|}\ktran_{1,2}g_n(k_1,k_2,x)\,dx
	\end{equation}
	for $k_1^2 + k_2^2 \not = k_0^2$.
  In the following, we suppress the arguments of $\phi$ and $\kappa$ for the sake of simplicity.
  The function $(k_1,k_2)\mapsto \abs{\kappa}^{-1} =\abs{k_0^2-k_1^2-k_2^2}^{-1/2}$ 
  is locally integrable, since we have for $c>k_0$ in cylindrical coordinates that
    $$
    \int_{\mathcal B_c} \abs{k_0^2-k_1^2-k_2^2}^{-1/2} d(k_1,k_2)
    =2\pi \int_0^c \rho\abs{k_0^2-\rho^2}^{-1/2} d\rho
    = 2\pi (k_0 + \sqrt{c^2-k_0}).
    $$
	Using this integrability and $\left|\kappa\right|\leq\left|\kappa_\epsilon\right|$, it follows that
	\begin{equation} \label{eq:upper1}
	\abs{\frac{\phi}{\kappa_\epsilon}\int_{\R}\e^{\i\kappa_\epsilon\left|r_3-x\right|}\ktran_{1,2}g_n(k_1,k_2,x)\,dx}\le \frac{\abs{\phi}}{\abs{\kappa}} \int_\R \abs{\ktran_{1,2} g_n(k_1,k_2,x)}\,dx\leq\frac{\left|\phi\right|}{2\pi\left|\kappa\right|}\lVert g_n\rVert_{L^1} \in L^1(\R^3).
	\end{equation}
	Taking into account \autoref{eq:point1} and \autoref{eq:upper1} and applying Lebesgue's dominated convergence theorem again, then gives
	\begin{equation*}
	\ktran_{1,2}u_n[\phi]= 
	\frac{\i}{2}\int_{\R^3}\frac{\phi}{\kappa}\int_{\R}\e^{\i\kappa\left|r_3-x\right|}\ktran_{1,2}g_n(k_1,k_2,x)\,dx\, d(k_1,k_2,r_3).
	\end{equation*}
	Next, we can express
	\begin{equation} \label{eq:expr}
	\begin{aligned}
	&\int_{\R} \e^{\i\kappa\left|r_3-x\right|}\ktran_{1,2}g_n(k_1,k_2,x) \, dx\\
	&= \e^{\i\kappa r_3}\int_{-\infty}^{r_3} \e^{-\i\kappa x} \ktran_{1,2}g_n(k_1,k_2,x) \, dx + \e^{-\i\kappa r_3} \int^{+\infty}_{r_3} \e^{\i\kappa x} \ktran_{1,2}g_n(k_1,k_2,x) \, dx \\
	&= \e^{\i\kappa r_3}\int_{\R} \left(1-H_{r_3}(k_1,k_2,x) \right)\e^{-\i\kappa x} \ktran_{1,2}g_n(k_1,k_2,x) \, dx\\
	&\quad +\e^{-\i\kappa r_3} \int_{\R} H_{r_3}(k_1,k_2,x) \e^{\i\kappa x} \ktran_{1,2}g_n(k_1,k_2,x) \, dx\\
	&= \e^{\i\kappa r_3}\int_{\R} \e^{-\i\kappa x} \ktran_{1,2}((1-H_{r_3})g_n)(k_1,k_2,x) \, dx + \e^{-\i\kappa r_3} \int_{\R} \e^{\i\kappa x} \ktran_{1,2}(H_{r_3} g_n)(k_1,k_2,x) \, dx\\
	&=\sqrt{2\pi}\left(\e^{\i\kappa r_3}\ktran((1-H_{r_3})g_n)(k_1,k_2,\kappa) +\e^{-\i\kappa r_3}\ktran(H_{r_3}g_n)(k_1,k_2,-\kappa)\right),
	\end{aligned}
	\end{equation}
	where we recall that for $k_1,k_2\in\R$ such that $k_1^2+k_2^2>k_0^2$, the analytic continuations of $\ktran((1-H_{r_3})g_n)$ and $\ktran(H_{r_3}g_n)$ to $\C^3$ have to be considered. 
	Therefore, we have
	\begin{equation} \label{eq:smoothpot}
  \begin{aligned}
	&\ktran_{1,2}u_n[\phi]\\&=\frac{\i\sqrt{\pi}}{\sqrt{2}}\int_{\R^3}\frac{\phi}{\kappa}\left(\e^{\i\kappa r_3}\ktran((1-H_{r_3})g_n)(k_1,k_2,\kappa) +\e^{-\i\kappa r_3}\ktran(H_{r_3}g_n)(k_1,k_2,-\kappa)\right)\,d(k_1,k_2,r_3).
  \end{aligned}
	\end{equation}
	We consider $n\to\infty$ in \autoref{eq:smoothpot} 
	and start with the right-hand side. 
	Taking into account $g_n\to g$ in $L^1(\R^3)$ and $\mathcal F: L^1(\R^3) \rightarrow C_0(\R^3)$, we obtain the pointwise limit
	\begin{equation}\label{eq:point2}
	\begin{aligned}
	&\lim_{n\to\infty}\frac{\phi}{\kappa}\left(\e^{\i\kappa r_3}\ktran((1-H_{r_3})g_n)(k_1,k_2,\kappa) +\e^{-\i\kappa r_3}\ktran(H_{r_3}g_n)(k_1,k_2,-\kappa)\right)\\
	&=\frac{\phi}{\kappa}\left(\e^{\i\kappa r_3}\ktran((1-H_{r_3})g)(k_1,k_2,\kappa) +\e^{-\i\kappa r_3}\ktran(H_{r_3}g)(k_1,k_2,-\kappa)\right)
	\end{aligned}
	\end{equation}
	for $k_1^2 + k_2^2 \not = k_0^2$. 
	Further, we have
		\begin{equation*}
	\left|\frac{\phi}{\kappa}\left(\e^{\i\kappa r_3}\ktran((1-H_{r_3})g_n)(k_1,k_2,\kappa) +\e^{-\i\kappa r_3}\ktran(H_{r_3}g_n)(k_1,k_2,-\kappa)\right)\right|\leq \frac{\left|\phi\right|}{(2\pi)^{3/2}\left|\kappa\right|}\lVert g_n\rVert_{L^1} 
	\end{equation*}
	for $k_1^2 + k_2^2 \not = k_0^2$,
	which follows from \autoref{eq:upper1} and \autoref{eq:expr}.
	As $g_n\to g$ in $L^1(\R^3)$, we can then find a constant $C>0$ and $N\in\mathbb{N}$ such that
	\begin{equation} \label{eq:upper2}
	\left|\frac{\phi}{\kappa}\left(\e^{\i\kappa r_3}\ktran((1-H_{r_3})g_n)(k_1,k_2,\kappa) +\e^{-\i\kappa r_3}\ktran(H_{r_3}g_n)(k_1,k_2,-\kappa)\right)\right|\leq C\frac{\left|\phi\right|}{\kappa}\lVert g\rVert_{L^1}\in L^{1}(\mathbb{R}^3)
	\end{equation}
	for every $n\geq N$ and almost every $(k_1,k_2,r_3)\in\mathbb{R}^3$. With \autoref{eq:point2} and \autoref{eq:upper2} at our disposal, Lebesgue's dominated convergence theorem gives
	\begin{equation}\label{eq:conv1}
	\begin{aligned}
	&\lim_{n\to\infty}\int_{\R^3}\frac{\phi}{\kappa}\left(\e^{\i\kappa r_3}\ktran((1-H_{r_3})g_n)(k_1,k_2,\kappa) 
	+\e^{-\i\kappa r_3}\ktran(H_{r_3}g_n)(k_1,k_2,-\kappa)\right)\,d(k_1,k_2,r_3)\\
	&=\int_{\R^3}\frac{\phi}{\kappa}\left(\e^{\i\kappa r_3}\ktran((1-H_{r_3})g)(k_1,k_2,\kappa) 
	+\e^{-\i\kappa r_3}\ktran(H_{r_3}g)(k_1,k_2,-\kappa)\right)\,d(k_1,k_2,r_3).
	\end{aligned}
	\end{equation}
	
	Next, we consider the convergence of the left-hand side in \autoref{eq:smoothpot}. 
	From \cite[Thm. 6, Rem. 1]{Gut04}, 
	it follows that the unique solution $u$ of \autoref{eq:Born2} for $g \in L^{q_1}(\R^3)$
	satisfying the Sommerfeld radiation condition, 
	fulfills
	\begin{equation} \label{eq:normest}
	\lVert u\rVert_{L^{q_2}}\leq C(k_0)\lVert g\rVert_{L^{q_1}}
	\end{equation}
	if $q_1 < \frac32$, $q_2 > 3$ and $\frac12 \le \frac{1}{q_1} - \frac{1}{q_2} \le \frac23$.
	In particular, 
	$q_2 = 3 + \varepsilon$ 
	and 
	$q_1 \in [9+3\varepsilon/(9+2\varepsilon),6+2\varepsilon/(5+\varepsilon)]$ 
	fulfill the above requirements for $\varepsilon >0$ small enough.
	Thus, for every $p>1$, we can find $1<q_1< p$ by setting $\varepsilon$ small enough.
	For $g \in L^p(\R^3)$, $p>1$ supported in $\mathcal B_r$, we know by the embedding in \autoref{embedding}
	that $g \in L^{q_1}(\R^3)$, for $1<q_1 < p$.
	Thus, \autoref{eq:normest} implies that $u_n\to u$ in $L^{q_2}(\R^3)$, since $g_n\to g$ in $L^{q_1}(\R^3)$. 
	In particular, $u_n\to u$ in $S^\prime(\R^3)$. 
	Then the continuity of $\mathcal{F}_{1,2}$ on $S^\prime(\R^3)$ gives
	\begin{equation}\label{eq:conv2}
	\begin{aligned}
	\mathcal{F}_{1,2}u[\phi]&=\lim_{n\to\infty}\mathcal{F}_{1,2}u_n[\phi]
	\end{aligned}
	\end{equation}
	for all $\phi \in S (\R^3)$ and by \autoref{eq:conv1} finally
	\begin{equation*}
	\begin{aligned}
	&\mathcal{F}_{1,2}u[\phi]
	\\&=\frac{\i\sqrt{\pi}}{\sqrt{2}}\int_{\R^3}\phi\underbrace{\frac{1}{\kappa}\left(\e^{\i\kappa r_3}\ktran((1-H_{r_3})g)(k_1,k_2,\kappa) +\e^{-\i\kappa r_3}\ktran(H_{r_3}g)(k_1,k_2,-\kappa)\right)}_{\in L^{1}_{\mathrm{loc}}(\R^3)}\,d(k_1,k_2,r_3).
	\end{aligned}
	\end{equation*}
	for all $\phi\in\mathcal{D}(\R^3)$. Then the assertion follows 
	by applying the du Bois--Reymond lemma, see \cite[Lem. 3.2]{Gru09}.
\end{proof}	

\subsection{Proof of \autoref{thm:jacobian}} \label{sec:jacobian}
In this subsection we compute the determinant of the Jacobian  $\left\lvert \nabla T_\pm (k_1,k_2,t) \right\rvert$.

\begin{proof}
By definition we have
\begin{equation} \label{eq:D}
	\left\lvert \nabla T_\pm (k_1,k_2,t) \right\rvert
	=
	\abs{\det \begin{pmatrix} \frac{\partial}{\partial k_1} R_{\bn(t),\alpha(t)}\bh, \frac{\partial}{\partial k_2} R_{\bn(t),\alpha(t)}\bh, \frac{\partial}{\partial t} R_{\bn(t),\alpha(t)}\bh\end{pmatrix}},
\end{equation}
where the first and second columns are given by
	\begin{equation*}
		\begin{aligned}
			\frac{\partial}{\partial k_1} R_{\bn,\alpha}\bh 
			=&  R_{\bn,\alpha}\frac{\partial}{\partial k_1}\bh 
			= R_{\bn,\alpha} \begin{pmatrix} 1\\ 0\\ \mp k_1/\kappa \end{pmatrix}, \\
			\frac{\partial}{\partial k_2} R_{\bn,\alpha}\bh 
			=&  R_{\bn,\alpha}\frac{\partial}{\partial k_2}\bh 
			= R_{\bn,\alpha} \begin{pmatrix} 0\\ 1\\ \mp k_2/\kappa \end{pmatrix}.
		\end{aligned}
	\end{equation*}
	Multiplying the argument of the determinant in \autoref{eq:D} 
	with the orthogonal matrix $R_{\bn(t),-\alpha(t)} = R_{\bn(t),\alpha(t)}^{-1}$ 
	does not change the determinant. 
	Hence we get
	\begin{equation}
		\begin{aligned}
			\left\lvert \nabla T_\pm (k_1,k_2,t) \right\rvert = {}&
			\abs{\det \begin{pmatrix} \frac{\partial}{\partial k_1}\bh, \frac{\partial}{\partial k_2}\bh, R_{\bn(t),-\alpha(t)}  \frac{\partial}{\partial t} R_{\bn(t),\alpha(t)}\bh\end{pmatrix}}
			\\
			= {}&
			\abs{\det \begin{pmatrix} 
					1&0&v_1\\
					0&1&v_2\\
					\mp\frac{k_1}{\kappa}&\mp\frac{k_2}{\kappa}& v_3
			\end{pmatrix}}
			\\ ={}&
			\abs{ \frac{k_1}{\kappa}v_1 + \frac{k_2}{\kappa}v_2 \pm v_3 },
		\end{aligned}
  	\label{eq:D-sum}
	\end{equation}
	where 
	\begin{equation*}
		\bv = (v_1,v_2,v_3)^\tT
		\coloneqq R_{\bn(t),-\alpha(t)} \,\frac{\partial}{\partial t} R_{\bn(t),\alpha(t)}\bh.
	\end{equation*}
	In the rest of this proof we calculate $\bv$.
	To shorten our notation, we suppress the dependency on $t$ and set
	$\mathrm{c} \coloneqq \cos\left(\alpha(t) \right)$ and $\mathrm{s} \coloneqq \sin\left(\alpha(t) \right)$.
	We set
	\begin{equation*}
		\begin{aligned}
			\bz\coloneqq 
			\frac{\partial}{\partial t} 
			R_{\bn(t),\alpha(t)}\bh 
			={} & 
			(1-\mathrm{c}) \left(\left(\bn'\cdot\bh\right) \bn 
			+ \left(\bn\cdot\bh\right) \bn'\right)
			+\alpha' \mathrm{s} \left( \left(\bn \cdot \bh \right)\bn- \bh \right)
			\\&
			- \alpha' \mathrm{c} \left( \bn \times \bh \right)
			- \mathrm{s} \left( \bn' \times \bh \right)
			,
		\end{aligned}
	\end{equation*}
	where the second equality is a consequence of Rodrigues' rotation formula
\begin{equation}\label{eq:rnk}
		R_{\bn,\alpha}\by
		=  (1-\cos(\alpha)) \left(\bn \cdot \by \right)\bn + \cos(\alpha)\, \by 
		- \sin(\alpha) \left( \bn \times \by \right).
\end{equation}
	Then we have
	\begin{equation*}
		\bv=
		R_{\bn(t),-\alpha(t)}\bz\\
		=
		(1-\mathrm{c}) \left(\bn \cdot \bz \right)\bn 
		+ \mathrm{c} \, \bz 
		+ \mathrm{s} \left( \bn \times \bz \right).
	\end{equation*}
	Since $\bn\cdot\bn=1$ and consequently $\bn\cdot\bn'=0$, 	we obtain
	\begin{equation*}
		\begin{aligned}
			\bv =
			{}& R_{\bn(t),-\alpha(t)}\bz
			=
			(1-\mathrm{c})\,\bn \left( (1-\mathrm{c})\left(\bn'\cdot\bh\right) 
			-\mathrm{s}\, \bn\cdot\left( \bn' \times \bh \right) \right)
			\\&
			+\mathrm{c} \left( 
			(1-\mathrm{c}) 
			\left(
			\left(\bn'\cdot\bh\right) \bn 
			+ \left(\bn\cdot\bh\right) \bn'\right)
			+ \alpha' \mathrm{s} \left( \left(\bn \cdot \bh \right) \bn- \bh \right)
			- \alpha' \mathrm{c} \left( \bn \times \bh \right)
			- \mathrm{s} \left( \bn' \times \bh \right) 
			\right)
			\\
			&
			+ \mathrm{s} \left( (1-\mathrm{c}) \left(\bn\cdot\bh\right) \left(\bn\times\bn'\right)
			- \alpha'\,\mathrm{s} \left(\bn\times\bh\right)
			-\alpha'\,\mathrm{c} \left(\bn\times\left(\bn\times\bh\right)\right)
			-\mathrm{s} \left(\bn\times \left(\bn'\times\bh\right)\right) 
			\right).
		\end{aligned}
	\end{equation*}
	Expanding the vector triple products in the last line using the Grassmann identity,
	we see that
	\begin{equation*}
		\begin{aligned}
			\bv
			={}&
			(1-\mathrm{c})\,\bn \left( (1-\mathrm{c})\left(\bn'\cdot\bh\right) 
			-\mathrm{s}\, \bn\cdot\left( \bn' \times \bh \right) \right)
			+\mathrm{c}  (1-\mathrm{c}) \left(\left(\bn'\cdot\bh\right) \bn 
			+ \left(\bn\cdot\bh\right) \bn'\right)\\
			&
			+\alpha'\, \mathrm{c} \, \mathrm{s} \left( \left(\bn \cdot \bh \right)\bn- \bh \right)
			- \alpha'\, \mathrm{c}^2 \left( \bn \times \bh \right) 
			- \mathrm{c} \, \mathrm{s} \left( \bn' \times \bh \right)
			+ \mathrm{s} (1-\mathrm{c}) \left(\bn\cdot\bh\right) \left(\bn\times\bn'\right)
			\\&
			- \alpha' \mathrm{s}^2 \left(\bn\times\bh\right)
			-\alpha'\mathrm{s}\, \mathrm{c} \left(\bn\left(\bn\cdot\bh\right) - \bh\right)
			-\mathrm{s}^2  \left(\bn'\left(\bn\cdot\bh\right)\right) .
		\end{aligned}
	\end{equation*}
	Sorting the terms, we obtain 
	\begin{equation*}
		\begin{aligned}
			\bv
			={}&
			\bn \left(
			(1-\mathrm{c})^2\left(\bn'\cdot\bh\right) 
			-\mathrm{s}\,(1-\mathrm{c}) \,\bn\cdot\left( \bn' \times \bh \right) 
			+\mathrm{c}\,(1-\mathrm{c}) \left(\bn'\cdot\bh\right)
			\right)
			\\&
			+\bn' \left(
			\mathrm{c} (1-\mathrm{c}) \left(\bn\cdot\bh\right)
			-\mathrm{s}^2 \left(\bn\cdot\bh\right)
			\right)
			- \alpha' \left( \bn \times \bh \right)
			- \mathrm{s} \, \mathrm{c} \left( \bn' \times \bh \right)
			+ \mathrm{s} (1-\mathrm{c}) \left(\bn\cdot\bh\right) \left(\bn\times\bn'\right)
			\\
			={}&
			\bn \left(
			(1-\mathrm{c})\left(\bn'\cdot\bh\right) 
			-\mathrm{s}\,(1-\mathrm{c}) \,\bn\cdot\left( \bn' \times \bh \right) 
			\right)
			-\bn' \left(1-\mathrm{c} \right) \left(\bn\cdot\bh\right)
			\\&
			- \alpha' \left( \bn \times \bh \right)
			- \mathrm{s} \, \mathrm{c} \left( \bn' \times \bh \right)
			+ \mathrm{s} (1-\mathrm{c}) \left(\bn\cdot\bh\right) \left(\bn\times\bn'\right).
		\end{aligned}
	\end{equation*}
	If $\bn'=\mathbf0$, then $\bv =  - \alpha' \left( \bn \times \bh \right)$.
	Otherwise, the vectors $\bn$, $\bn'$ and $\bn\times\bn'$ are orthogonal and Lagrange's identity yields
	\begin{align*}
		\bn'\times\bh
		&= \bn\left(\bn\cdot \left(\bn'\times\bh\right) \right) 
		+ \frac{1}{\norm{\bn\times\bn'}^2}\left(\bn\times\bn'\right)
		\left(\left(\bn\times\bn'\right) \cdot \left(\bn'\times\bh\right) \right)
		\\
		&= \bn\left(\bn\cdot \left(\bn'\times\bh\right) \right) 
		- \left(\bn\times\bn'\right)
		\left(\bn\cdot \bh \right) 		.
	\end{align*}
	Hence, we obtain 
	\begin{align}
			\bv
			={}&
			\bn \left(
			(1-\mathrm{c})\left(\bn'\cdot\bh\right) 
			-\mathrm{s}\,(1-\mathrm{c}) \,\bn\cdot\left( \bn' \times \bh \right) 
			\right)\notag \\
			& -\bn' \left(1-\mathrm{c} \right) \left(\bn\cdot\bh\right)
			- \alpha' \left( \bn \times \bh \right)
      - \mathrm{s} \, \mathrm{c} \left( \bn\left(\bn\cdot \left(\bn'\times\bh\right) \right) 
			-		
			\left(\bn\times\bn'\right)
			\left(\bn\cdot \bh \right)  \right)
			\notag \\&
			+ \mathrm{s} \left(1-\mathrm{c} \right) 
			\left(\bn\cdot\bh\right) \left(\bn\times\bn'\right)
			\notag \\
			=
			{}&
			(1-\mathrm{c})
      \left(
			\bn\, (\bn'\cdot\bh) 
      -\bn'  \left(\bn\cdot\bh\right)
      \right)
			-\bn\, \mathrm{s} \left(\bn\cdot\left( \bn' \times \bh \right) \right)
			- \alpha'\,\left( \bn \times \bh \right) 
			+  \mathrm{s} \left(\bn\cdot\bh\right) \left(\bn\times\bn'\right)
			.
  	\label{eq:Rds}
	\end{align}
  By \autoref{eq:D-sum}, we have
  \begin{equation*}
   \left\lvert \nabla T_\pm (k_1,k_2,t) \right\rvert
    = \abs{ 
    \begin{pmatrix}v_1\\v_2\\v_3\end{pmatrix}
    \cdot \begin{pmatrix} k_1/\kappa\\ k_2/\kappa\\ \pm 1\end{pmatrix}  }
    = \frac1\kappa \abs{ 
    \begin{pmatrix}v_1\\v_2\\v_3\end{pmatrix}
    \cdot \begin{pmatrix} k_1\\ k_2\\ \pm\kappa\end{pmatrix}  }
    = \frac1\kappa \left\lvert \bv\cdot \left(\bh + \begin{pmatrix}0\\0\\k_0\end{pmatrix} \right) \right\rvert
    = \frac{k_0}{\kappa} \abs{v_3},
  \end{equation*}
  where the last equality follows from 
  $
  \bv \cdot \bh = 0.
  $
  Replacing $v_3$ by the third component of \autoref{eq:Rds} yields
  \begin{equation*}
  \begin{aligned}
    \left\lvert \nabla T_\pm (k_1,k_2,t) \right\rvert
    = \frac{k_0}{\kappa}
    {}& \big|
    \left( (1-\mathrm{c}) \left(
    n_3\, 
    (\bn'\cdot\bh) 
    -n_3'  \left(\bn\cdot\bh\right)
    \right)
    -n_3\mathrm{s} \,\bn\cdot\left( \bn' \times \bh \right) 
    \right)
    \\&
    - \alpha' \left( n_1k_2-n_2k_1 \right) 
    +  \mathrm{s} \left(\bn\cdot\bh\right) \left(n_1n_2'-n_2n_1'\right)
    \big|
    ,
  \end{aligned}
  \end{equation*}
  which proves the first assertion. The Jacobian determinant is in $L^1(\mathcal{U})$, because $(k_1,k_2)\mapsto 1/\kappa(k_1,k_2)$ is locally integrable on $\R^2$ while the remaining expression is bounded on $\mathcal{U}$.
\end{proof}

\subsection*{Acknowledgements} This work is supported by the Austrian Science Fund (FWF) within SFB F68 (``Tomography
across the Scales''), Projects F68-06 and F68-07. 
G.S. acknowledges funding by the DFG under Germany's Excellence Strategy – The Berlin Mathematics Research Center MATH+ (EXC-2046/1,  Projektnummer:  390685689).
\section*{References}
\renewcommand{\i}{\ii}
\printbibliography[heading=none]

\end{document}